%% file: main.tex
\documentclass[11pt]{article}

\input{macros}

\usepackage[titles]{tocloft}
\author{Kiril Bangachev\thanks{Dept. of EECS, MIT. \texttt{kirilb@mit.edu}. Supported by a Siebel Scholarship.} \quad Guy Bresler\thanks{Dept. of EECS, MIT. \texttt{guy@mit.edu}. Supported by NSF Career Award CCF-1940205.}}

\title{Sandwiching Random Geometric Graphs and Erd\H{o}s-R\'enyi with Applications:
Sharp Thresholds, Robust Testing, and Enumeration}

\providecommand{\keywords}[1]
{
  \small	
  \textbf{\textit{Keywords---}} #1
}

\begin{document}

\maketitle

\pagenumbering{gobble}

\begin{abstract} 
The distribution $\mathsf{RGG}(n,\mathbb{S}^{d-1},p)$ is formed by sampling independent vectors $\{V_i\}_{i = 1}^n$ uniformly on $\mathbb{S}^{d-1}$ and placing an edge between pairs of vertices $i$ and $j$ for which $\langle V_i,V_j\rangle \ge \tau^p_d,$ where $\tau^p_d$ is such that the expected density is $p.$ Our main result is a poly-time implementable coupling between Erd\H{o}s-R\'enyi and $\mathsf{RGG}$ such that $\mathsf{G}(n,p(1 - \Tilde{O}(\sqrt{np/d})))\subseteq \mathsf{RGG}(n,\mathbb{S}^{d-1},p)\subseteq \mathsf{G}(n,p(1 + \Tilde{O}(\sqrt{np/d})))$ edgewise with high probability when $d\gg np.$ We apply the result to:

1) \emph{Sharp Thresholds:} We show that for any monotone property having a sharp threshold with respect to the Erd\H{o}s-R\'enyi distribution and critical probability $p^c_n,$ random geometric graphs also exhibit a sharp threshold when $d\gg np^c_n,$ thus partially answering a question of Perkins. 

2) \emph{Robust Testing:} The coupling shows that testing between $\mathsf{G}(n,p)$ and $\mathsf{RGG}(n,\mathbb{S}^{d-1},p)$ with $\epsilon n^2p$ adversarially corrupted edges for any constant $\epsilon>0$ is information-theoretically impossible when $d\gg np.$ We match this lower bound with an efficient (constant degree SoS) spectral refutation algorithm when $d\ll np.$

3) \emph{Enumeration:} We show that the number of geometric graphs in dimension $d$ is at least $\exp(dn\log^{-7}n)$, recovering (up to the log factors) the sharp result of Sauermann. 

\end{abstract}


\newpage
\setcounter{tocdepth}{2}
\tableofcontents

\clearpage
\pagenumbering{arabic}

\input{IntroductionV5}

\input{Preliminaries}

\input{NonTechnical}

\input{Coupling}

\input{RobustTesting}

\input{Enumeration}

\input{SharpThresholds}

\input{Discussion}

\section*{Acknowledgements}
We are grateful to Sidhanth Mohanty for many helpful discussions on robust statics and random geometric graphs. We also thank Chenghao Guo for raising the question about efficient adversaries for the robust testing problem. We are indebted to Dheeraj Nagaraj for many conversations on random geometric graphs over the years and also specifically for discussions about the flip orientation map. 

\bibliography{bibliography}
\bibliographystyle{alpha}

\appendix

\input{SphericalDistribution}
\input{GraphCountingAppendix}

\input{RGGRepresentation}
\end{document}

%% file: macros.tex
\usepackage[utf8]{inputenc}
\usepackage{amsmath}
\usepackage{amssymb,amsfonts,amsthm}
\usepackage{mathrsfs}
\usepackage{mathtools}
\usepackage{algorithm}
\usepackage{algorithmic}
\usepackage{xcolor}
\usepackage{tikz}
\usepackage{fullpage}
\usepackage{dsfont}
\usepackage{thmtools}
\usepackage{thm-restate}
\usepackage{enumitem}
\usepackage{dsfont}
\usepackage{graphicx}
\usepackage{sidecap}
\usepackage{wrapfig}

\usepackage{accents}

\usepackage{hyperref}
\usepackage{xcolor}
\hypersetup{
  colorlinks   = true, 
  urlcolor     = blue, 
  linkcolor    = blue, 
  citecolor   = blue 
}
\usepackage[capitalize, nameinlink]{cleveref}

\usepackage{tikz}
\usetikzlibrary{decorations.pathreplacing}
\usetikzlibrary{arrows,positioning}
\DeclareRobustCommand\shape{
 \lower5pt\hbox{
 \hskip-7pt
  \tikzset{circ/.style={circle, draw, fill=black, scale=.15}}
  \begin{tikzpicture}[semithick,scale=.3]
  \node (l1) at (0,.5) [circ]{};
  \node (l3) at (0.5,0.3) [circ]{};
  \draw[-] (l1) to node [auto] {} (l3);
    \end{tikzpicture}
  \hskip-8pt}
}

\makeatletter
\newcommand{\mylabel}[2]{#2\def\@currentlabel{#2}\label{#1}}
\makeatother

\renewcommand*{\big}[1]{\vcenter{\hbox{\scalebox{1.3}{\ensuremath#1}}}}

\newcommand{\unif}{\mathsf{Unif}}

\newtheorem{theorem}{Theorem}[section]

\newtheorem{lemma}{Lemma}[section]

\newtheorem{corollary}[lemma]{Corollary}
\newtheorem{observation}[lemma]{Observation}

\newtheorem{proposition}[theorem]{Proposition}
\theoremstyle{definition}
\newtheorem{remark}{Remark}[]
\newtheorem{definition}{Definition}
\newtheorem{algorithms}{Algorithm}
\newtheorem{problem}{Problem}[]

\newcommand{\sphericaloned}{\mathcal{D}_d}

\newcommand{\trace}{\mathsf{tr}}

\newcommand{\RGGsphere}{{\mathsf{RGG}(n,\dsphere, p)}}

\newcommand{\RGGgauss}{{\mathsf{RGG}(n,\mathcal{N}(0, \frac{1}{d}I_d), p)}}

\newcommand{\var}{\mathbb{V}}

\DeclareMathOperator*{\E}{{\rm I}\kern-0.18em{\rm E}}

\newcommand{\prob}{\Pr}
\newcommand{\expect}{\E}
\renewcommand{\Pr}{\,{\rm I}\kern-0.18em{\rm P}}

\newcommand{\TV}{{d_\mathsf{TV}}}

\newcommand{\law}{\mathcal{L}}

\newcommand{\bfG}{\mathbf{G}}
\newcommand{\bfH}{\mathbf{H}}

\newcommand{\RGG}{\mathsf{RGG}}

\newcommand{\cF}{\mathcal{F}}

\newcommand{\ER}{{Erd\H{o}s-R\'enyi}}
\newcommand{\ERspace}{{Erd\H{o}s-R\'enyi }}

\newcommand{\polylog}{\mathsf{polylog}}

\newcommand{\dsphere}{\mathbb{S}^{d-1}}

\newcommand{\iidsim}{\stackrel{\mathrm{i.i.d.}}\sim}

\newcommand{\ergraph}{\mathsf{G}(n,p)}

\newcommand{\Bernoulli}{\mathsf{Bern}}

\newcommand{\indicator}{\mathds{1}}

\newcommand{\Vfin}[1]{\widetilde{V}_{#1}}

\newcommand{\coloneqq}{:=}

%% file: IntroductionV5.tex
\section{Introduction}

Random geometric graphs have emerged as a fruitful way to model dependence between edges in random networks in both theory and practice \cite{penrose03,serrano2008self,boguna2021network,duchemin22}.    
Their defining property is 
that edges are included based on similarity 
of feature vectors associated to the nodes. 
Geometric graphs 
accurately capture certain properties of  social networks \cite{mcfarland1973social,mcpherson2001birds,hoff2002latent,Preciado2009SpectralAO,ESTRADA201620}, biological networks \cite{higham08,vempala23kcap,szklarczyk2021string}, and wireless communication networks \cite{gilbert61planenetworks,haenggi_2012}.

A simple example is given by the \emph{Spherical Random Geometric Graph}, popular in the mathematical community due to its simplicity and elegance.

\begin{definition}[Spherical Random Geometric Graph]
\label{def:sphericalrgg}
A sample from the distribution $\RGGsphere$ over graphs on $n$ vertices $[n] = \{1,\ldots, n\}$ of
dimension $d$ with expected density $p$ is generated as follows. First, $n$ vectors $V_1,  \ldots, V_n$ are drawn independently from the uniform distribution on the sphere $\dsphere.$ Then, an edge between $i$ and $j$ is formed if and only if $\langle V_i ,V_j\rangle \ge \tau^p_d.$ Here, $\tau^p_d$ is chosen so that $p = \prob[\langle V_i ,V_j\rangle \ge \tau^p_d].$ 
\end{definition}

 Much is known about the theoretical properties of random geometric graphs in the low-dimensional regime when $d$ is constant or 
 on the order of $O(\polylog(n))$ \cite{penrose03}. A recent breakthrough work in this setting  \cite{Liu22Expander} shows that in certain dimensions $d = \Theta(\log n), $ random geometric graphs are sparse two-dimensional expanders, an object for which no randomized construction was previously known. 

The interest in the high-dimensional regime when $d\ggg \polylog(n)$ is much more recent.\footnote{We use $a\ggg b$ to denote $a\ge \log (n)^c b$ for some fixed, implicit, constant $c$. Similarly for $\lll.$} A strikingly simple, but fundamental, question remains open after more than a decade of continuous work and improvements: \emph{For what sequences of $n,d,p$ are the edges of $\RGGsphere$ actually dependent?} When $d$ is large enough (as a function of $n$ and $p$), the $\RGGsphere$ distribution is only $o(1)$ total variation distance away from $\ergraph.$ In other words, for large enough $d,$ the edges are asymptotically independent \cite{Devroye11}. The majority of the work on high-dimensional random geometric graphs has focused on determining the precise dimension for which this occurs \cite{Devroye11,Bubeck14RGG,Brennan19PhaseTransition,Liu2022STOC,bangachev2024fourier}. When $d\lll (np)^3,$ a simple triangle-based statistic witnesses edge dependence. Conjecturally, when $d\ggg (np)^3,$ $\RGGsphere$ converges to $\ergraph$ in total variation distance. 

Insofar as the purpose of the $\RGGsphere$ distribution is to model edge dependence, the interesting regime is $d\lll (np)^3$, where the distribution is far from \ERspace and the edges are meaningfully dependent. Yet, very little is known in the ``bulk'' $\polylog(n)\lll d\lll (np)^3$ of this regime. This leads to the main motivation of the current paper:  
\begin{center}
    \textit{What can be said about $\RGGsphere$ in the regime $\polylog(n)\lll d\lll (np)^3$?}
\end{center}
The higher the dimension, the weaker the dependence is between edges and, thus, one can expect behaviour more similar to \ER.
One well-studied property in this regime is the spectrum of the adjacency matrix \cite{Liu2021APV,li2023spectral,bangachev2024fourier}. When $d\lll np,$ the centered adjacency matrix of a sample from $\RGGsphere$ has an operator norm of $\tilde{\Theta}(np/\sqrt{d})$ which is significantly larger than that of $\ergraph$ at $\tilde{\Theta}(\sqrt{np}).$ Conversely, when $d\ggg np,$ the operator norm is $\tilde{\Theta}(\sqrt{np}),$ which is on the same order of $\ergraph.$ A different result is that when $d\ggg \polylog(n),$ the clique number of $\RGGsphere$ is essentially the same as that of $\ergraph$ \cite{Devroye11}. Aside from existence of a clique of a given size, the thresholds in $p$ for other natural properties such as connectivity and Hamiltonicity are completely open when $\polylog(n)\lll d \lll (np)^3.$ Is the behaviour of $\RGGsphere$ with respect to those properties similar to the behaviour of the much better-understood $\ergraph$ distribution? 

\subsection{Our Angle}
In order to understand which properties might be shared between the random geometric graph and \ERspace in the regime $\log (n)\lll d \lll (np)^3$, we aim to show a new style of comparison result between the two graph models.  
Of course, such a comparison should \emph{not be in total variation}, because in this regime 
the two distributions have total variation distance $1 - o(1)$. Large total variation is a consequence of differing signed triangle counts, but what about other properties? It seems plausible that the distributions share many properties in addition to their clique number (which was studied in \cite{Devroye11}). 

Rather than studying properties individually, we aim to simultaneously address all \emph{edge-monotone properties}. These include existence of a clique of a given size, but also Hamiltonicity, connectivity, existence of a perfect matching and many more. One of the most fundamental facts about $\ergraph$ is that \emph{monotone properties} subject to very general criteria exhibit a sharp-threshold in $p$ \cite{Friedgut1999SharpTO}. Yet, extending this result to non-product measures has proven to be a challenging task. Important progress has been made in the case of measures satisfying the FKG lattice condition \cite{grimmett2006random}. However, the RGG does not satisfy this condition (see \cref{prop:failureofFKG}). \cite{goel05monotoneproperties} prove sharp thresholds for monotone properties for random geometric graphs but in the low-dimensional regime $o(\log n).$ As addressing monotone properties is one of our main goals, we aim for a comparison with \ERspace under which similar distributions assign similar probabilities to monotone events.

Besides exploring the similarities between $\ergraph$ and $\RGGsphere,$ we will also aim for a comparison that evinces \emph{meaningful} differences between the two distributions. 
Real-world networks typically 
do not exactly match any natural generative model such as the one in \cref{def:sphericalrgg}, and model misspecification can be captured by allowing for a small proportion of the edges to be adversarially corrupted. 
While the signed triangle count distinguishes the two graph models when $d\lll (np)^3,$ this test turns out to be rather brittle. 
\emph{When are there tests between $\RGGsphere$ and $\ergraph$ that are robust to edge corruptions?}
 
The combination of \emph{monotonicity} and \emph{edge corruptions} leads us to the following notion of comparison with \ER: approximate stochastic dominance. 
Our goal will be to
\begin{quote}\vspace{3mm}
   \mbox{ \parbox{\linewidth}{
   \emph{Couple $H_-\sim \mathsf{G}(n,p(1-\epsilon)), H_+\sim \mathsf{G}(n,p(1+\epsilon)),$ and $G\sim \RGGsphere$ for some vanishingly small $\epsilon$ in such a way that 
   $H_-\subseteq G\subseteq H_+$ with high probability (where inclusion is with respect to the edge sets). 
   }} }\vspace{3mm}
\end{quote}
Our main result
is that such a coupling exists if and only if $d\ggg np.$ (\cite{Liu2022STOC} also exhibit such a coupling but in the weaker regime $d\ggg n^2p^2$. Furthermore, their result is not algorithmic while ours is. A more detailed comparison is made in \cref{sec:priorwork}.)
This has several applications: 
\begin{enumerate}
    \item \emph{Sharp Thresholds:} We use the \emph{edge-monotonicity} of the coupling to compare monotone properties of $\ergraph$ and $\RGGsphere.$ We give the first proof of sharp thresholds for monotone properties of $\RGGsphere$ for a polynomially large range of dimensions where $\RGGsphere$ is TV-distance $1-o(1)$ away from $\ergraph$. Furthermore, we show that the critical probabilities are nearly identical with those of $\ergraph$ in this range. 
    \item \emph{Robust Testing:} One interpretation of the approximate stochastic dominance is that a sample of $\ergraph$ is a sample of $\RGGsphere$ with $O(\epsilon n^2 p)$ adversarially corrupted edges. This leads us to revisit the well-studied question of testing between $\ergraph$ and $\RGGsphere$ 
    in the previously unexplored setting with adversarial edge-corruptions. Our comparison result shows impossibility of robust testing for $d\ggg np$ and in the complementary regime we provide an efficient sum-of-squares algorithm.
    \item \emph{Enumeration:} The approximate stochastic dominance implies a lower bound on the support size of the $\RGGsphere$ distribution. This allows us to use $\RGGsphere$ as a randomized construction for 
    counting the number of geometric graphs in dimension $d.$ 
\end{enumerate}
We next state our main coupling result and then describe these applications in more detail. 

\subsection{Main Coupling Theorem}

\begin{theorem}[Main Coupling Result]
\label{thm:algorithmiccoupling}
Consider some $n,d,p$ such that $d= \Omega( \max(np,\log n))$ and $p = \Omega(1/n), p\le 1/2.$ There exists a polynomial-time algorithm which on input $H\sim \ergraph,$ outputs $n$ vectors $V_1 = V_1(H), V_2 = V_2(H), \ldots, V_n = V_n(H)$ in $\dsphere$ with the following two properties:
\begin{enumerate}
    \item The marginal distribution of the vectors is independent uniform: $V_1, V_2, \ldots, V_n\iidsim\unif(\dsphere).$
    \item With high probability (over the input and internal randomness), $\indicator[\langle V_i, V_j\rangle\ge \tau^p_d] = H_{ij}$ for all pairs $i,j$ such that $|\langle V_i, V_j\rangle - \tau_p^d|\ge \frac{c\max(\sqrt{np}, \sqrt{\log n})(\log n)^{3/2}}{d}$ for some absolute constant $c.$
     
\end{enumerate}
\end{theorem}
In other words, the algorithm described above couples $\RGGsphere$ and $\ergraph$ in such a way that they only differ at pairs $(i,j)$ for which $\langle V_i, V_j\rangle $ is very close to the threshold $\tau^p_{d}$. Elementary calculations shows that with high probability, $G,H$ differ at only $\tilde{O}\big(\binom{n}{2}p\times\sqrt{\frac{{np}}{d}}\big) = o\big(\binom{n}{2}p\big)$ edges. 

\begin{remark}
\label{rmk:representation}
\cref{thm:algorithmiccoupling} shows that one can represent $\RGGsphere$ as an \ERspace graph with few ``defect edges.'' By performing multiple rounds of our coupling algorithm, we can sequentially ``fix'' these defect edges. 
We obtain a curious novel representation of $\RGGsphere$ by recursively planting \ERspace graphs (of exponentially shrinking size) in $\ergraph.$ When $d = n^{1+\epsilon}p,$ the recursion only takes constantly many rounds. See  
\cref{sec:recursiverepresentation} for the representation result.
\end{remark}

An immediate implication of \cref{thm:algorithmiccoupling} is the following approximate  stochastic dominance. 

\begin{corollary}[Approximate Stochastic Dominance]
\label{thm:stochasticcoupling}
Consider some $n,d,p$ such that $d= \Omega(np\log^4 n, \log^5n)$ and $p = \Omega(1/n).$ Then, there exists a coupling of $G\sim\RGG(n,\dsphere,p),$ $H_-\sim \mathsf{G}(n,p(1-\frac{c\max(\sqrt{np}, \sqrt{\log n})(\log n)^{2}}{\sqrt{d}}))$ and $ H_+\sim \mathsf{G}(n,p(1+\frac{c\max(\sqrt{np}, \sqrt{\log n})(\log n)^{2}}{\sqrt{d}})),$ 
 such that 
$
H_-\subseteq G \subseteq H_+
$
with probability $1-o(1)$, where inclusion is defined as inclusion over the edge set.
\end{corollary}

In \cite{Liu2022STOC}, the authors prove \cref{thm:stochasticcoupling} under the weaker condition $d \ggg n^2p^2$ with the purpose of describing the local structure of random geometric graphs. Our result, which holds for $d\ggg np$ is based on an entirely different proof technique. The difference between $np$ and $(np)^2$ is not merely quantitative. Qualitatively, $d = \tilde{\Theta}(np)$ is of fundamental importance to spherical random geometric graphs as it is simultaneously a spectral and an entropic threshold, see \cref{sec:intepretationspectralentropic}. This dual interpretation makes it key to our applications and using $d\ggg n^2p^2$ would give highly suboptimal results for robust testing and enumeration. 


\begin{remark} We note that in general a relaxed stochastic dominance of the form in \cref{thm:stochasticcoupling} does not guarantee closeness in total variation even if one couples 
$H_-\sim \mathsf{G}(n,p), H_+\sim \mathsf{G}(n,p),$ and $ G\sim\mathcal{D}$ such that $H_-\subseteq G\subseteq H_+$ with high probability for some $\mathcal{D}.$ Specifically, consider a distribution $\mathcal{D}$ defined by the following coupling. $H_-$ is drawn from $\ergraph.$ Then, $H_+$ is a copy of $H_-$ except that a uniformly random non-edge is changed to an edge. One can check that whenever $p= \Omega(1/n),$ the distribution of $H_+$ is only $o(1)$ TV distance away from $\ergraph.$ Finally, $G$ is the graph among $H_-,H_+$ with an even number of edges. Clearly, $H_-\subseteq G\subseteq H_+$  w.h.p. but $\TV(\mathcal{D}, \ergraph) \ge 1/2 - o(1)$ as the distribution $\mathcal{D}$ is supported on graphs with an even number of edges.
\end{remark}

\subsection{Sharp Thresholds} 

The question of whether monotone properties exhibit sharp thresholds is by far best understood for product distributions.\footnote{See \cref{sec:preliminaries} for the formal definitions of sharp thresholds, critical probability, and critical window.} The seminal paper of \cite{Friedgut1999SharpTO} building on decades of work on the subject \cite{margulis1974,Russo1981,RussoLucio1982Aazl,bourgain92influenceproduct,talagrand94onrusso,FriedgutEhud1996Emgp} establishes very broad conditions on vertex-transitive properties for which the \ERspace distribution exhibits sharp-thresholds. In non-product settings, much less is understood. Important progress has been made in the case of positive measures satisfying the FKG lattice property \cite{grimmett2006random,grimmett06fkg,grimett11sharpthresholds}. However,
this property does not hold for random geometric graphs, which we check in~\cref{prop:failureofFKG}. 

In the case of low-dimensional random geometric graphs (over the solid cube $[0,1]^d$), \cite{goel05monotoneproperties} show that monotone properties have \emph{additive} sharp-thresholds by proving that if $G_1$ and $G_2$ are independent samples with densities $p(1-o(1))$ and $p$, then $G_1\subseteq G_2$ with high probability after an appropriate vertex relabelling.
Also in low dimension, but in a somewhat different setting where the vertices come from a Poisson process and hence the random geometric graph is of variable size, \cite{Bradonjic2013OnST} gives necessary conditions for properties monotone in \emph{both} edges and vertices that do not exhibit \emph{multiplicative} sharp thresholds. Under certain compatibility of the vertex and edge thresholds, they can transfer the results to the fixed-size edge-monotone setting. However,  
 a generic result for the fixed-size edge-monotone setting is only conjectured in \cite{Bradonjic2013OnST}.
The methods of \cite{goel05monotoneproperties,Bradonjic2013OnST} do not seem applicable to the high-dimensional setting because both rely on the fact that a very fine (grid) subdivision of the space is well-covered by the latent vectors. In $d$ dimensions, the subdivision has $\exp(\Omega(d))$ cells, which is much larger than $n$ when $d = \omega(\log n).$ 

This suggests a natural question, posed by Perkins \cite{perkins2024searching}: \emph{Do random geometric graphs in higher dimensions exhibit sharp thresholds?} Of course, when the dimension is sufficiently large, the $\RGGsphere$ distribution is asymptotically the same as $\ergraph$ and the answer is ``yes''. But what about the case when $d\lll n^3p^3$ and the two distributions are provably far in total variation?
 

Our coupling \cref{thm:stochasticcoupling} in the regime $d\ggg np$ between random geometric graph $G$ and \ERspace $H_-$ and $ H_+$ is monotone in the sense that 
$
H_-\subseteq G \subseteq H_+.
$
This implies that 
whenever 
$H_-\sim \mathsf{G}(n,p(1-\frac{c\max(\sqrt{np}, \sqrt{\log n})(\log n)^{2}}{\sqrt{d}}))$ satisfies an increasing property with high probability, so does $G\sim \RGGsphere,$ and whenever 
$G\sim \RGGsphere$ satisfies an increasing property, so does $H_+\sim \mathsf{G}(n,p(1+\frac{c\max(\sqrt{np}, \sqrt{\log n})(\log n)^{2}}{\sqrt{d}})).$ 
This allows us to translate sharp-threshold phenomena from \ERspace to RGG. We formalize this argument in \cref{sec:proofofsharptgresholds} and record it as the following theorem.

\begin{theorem}
\label{thm:sharpthresholds}
Suppose that $\mathcal{P}_n$ is a monotone property with critical probability $p^c_n$ in $\ergraph$, which satisfies $p^c_n = \Omega(1/n).$ 
Suppose further that $\mathcal{P}_n$ exhibits a sharp threshold with respect to $\ergraph.$  
If $d_n=  \omega(\max(np^c_n\log^4 n, \log^5n)),$ the property $\mathcal{P}_n$ also exhibits a sharp threshold for $\RGG(n,d_n,p)$ and has critical probability $p^c_n (1 + o_n(1)).$
\end{theorem}

\subsection{Robust Testing} 
The problem of testing between \ERspace and a spherical random geometric graph has attracted significant attention in recent years \cite{Devroye11,Bubeck14RGG,Brennan19PhaseTransition,Liu2022STOC,bangachev2024fourier}. Formally, one observes a graph $G$ and needs to decide $H_0:G\sim \ergraph$ or $H_1:G\sim \RGGsphere.$ By comparing the number of signed triangles $\sum_{i<j<k}(G_{ij} - p)(G_{jk}-p)(G_{ki}-p)$ to a threshold, one can succeed with high probability whenever $d = \omega(n^3p^3\log^3(1/p))$ \cite{Bubeck14RGG,Schramm_2022}.  The current state-of-the-art information-theoretic lower bounds are that when $d\ge (\log n)^C n^3p^2,$ the problem is information theoretically impossible as well as when 
$p = \Theta(1/n)$ and $d\ge (\log n)^C$ for some constant $C$ \cite{Liu2022STOC}. Even ignoring $\polylog $ factors, there is still a $1/p$ gap between the testing algorithm and the information-theoretic lower bound when $1/n\lll p\lll 1/2.$ This gap was recently closed with respect to low-degree polynomial tests by \cite{bangachev2024fourier}, i.e., counting signed triangles is optimal among all degree $\log^{2-\epsilon} n$ polynomial tests.

The existing methods in the literature for both the upper and lower-bounds for feasibility of testing are extremely brittle -- they 
fail even under a very small fraction of adversarially corrupted edges. On the algorithmic side, 
consider the well-studied test mentioned above based on comparing the number of signed triangles to a threshold. The expected number of signed triangles in $\RGGsphere$ is $\tilde{\Theta}(n^3p^3/\sqrt{d})$ \cite{Liu2022STOC} while in $\ergraph$ it is 0. 
Thus, planting a clique of size $\tilde{O}(np/d^{1/6})$ in $\ergraph$ (which requires $\tilde{O}(n^2p^2/d^{1/3}) = o(n^2p)$ adversarial corruptions) makes the expected counts match and hence the test fails. One can argue similarly for other tests based on counting subgraphs. Regarding lower-bounds, all of the prior works in the setting without adversarial corruptions \cite{Bubeck14RGG,Brennan19PhaseTransition,Liu2022STOC} are based on bounding total variation distance and strongly exploit distributional assumptions.
We do not expect that the techniques in any of those papers apply to the robust testing setting.

The methods of the current work not only apply to, but also give tight upper and lower bounds, in a setting of a "maximum possible" fraction of corruptions. We work in the strong contamination model with edge corruptions, which has been extensively studied for the stochastic block model, e.g. \cite{banks21lsh} for testing and \cite{montanari16sdpsbm,ding21robustsbmrecovery,mohanty24robustsbm} for recovery.
\begin{problem}[Robust Testing of Geometry With Edge Corruptions]
\label{problem:robusttesting}
Let $n,d\in \mathbb{N},$ $p\in [0,1]$ and $\epsilon>0$ be a fixed constant. 
One observes a graph $H$ and needs to decide between the two hypotheses $H_0: H = \mathcal{A}(\bfG),\bfG\sim \ergraph$ and 
$H_1: H = \mathcal{A}(\bfG),\bfG\sim \mathsf{RGG}(n,d,p).$ Here, $\mathcal{A}:\{0,1\}^{\binom{n}{2}}\longrightarrow \{0,1\}^{\binom{n}{2}}$ is any unknown mapping such that 
$\mathcal{A}(G)$ and $G$ differ in at most $\epsilon \binom{n}{2}p$ entries for any adjacency matrix $G\in \{0,1\}^{\binom{n}{2}}.$ 
\end{problem}  

We restrict the adversary to changing at most $\epsilon \binom{n}{2}p$ edges for a small constant $\epsilon$ as otherwise \cref{problem:robusttesting} is trivially impossible. To see this, note that both $\mathsf{RGG}(n,d,p)$ and 
$\ergraph$ have with high probability at most $1.5\binom{n}{2}p$ edges. Hence, an adversary $
\mathcal{A}$ who simply deletes all edges of a sample from $\RGG$ and then inserts Bernoulli $p$ edges (in total $3\binom{n}{2}p$ w.h.p.) can transform $\RGGsphere$ into a sample from $\ergraph.$ Thus, \cref{problem:robusttesting} captures the setting of ``maximal adversaries.'' 

Our first result is an information-theoretic lower bound which immediately follows from \cref{thm:algorithmiccoupling}.
Whenever $d\ge K(\epsilon)\max(np (\log n)^{4}, (\log n)^5)$ (where $K(\epsilon)$ depends solely on $\epsilon$), one can with high probability (adversarially) corrupt $\epsilon \binom{n}{2}p$ edges of a sample from $\RGGsphere$ and produce a sample from $\ergraph.$
Hence, the robust testing problem is information-theoretically impossible when 
$d\ge K(\epsilon)\max(np (\log n)^{4}, (\log n)^5).$ When $d = o((np)^3(\log 1/p)^3), d\ge K(\epsilon)\max(np (\log n)^{4}, (\log n)^5)),$ the problem is information-theoretically impossible even against a polynomial-time adversary. In this regime, an adversary can see a sample $H$ and (correctly with high probability) test using the signed-triangle statistic \cite{Liu2022STOC} whether the sample is from $\ergraph$ or $\RGGsphere.$ If $H$ is from $\ergraph,$ the adversary can use the efficient algorithm from \cref{thm:algorithmiccoupling} and produce (with high-probability) a sample from $\RGGsphere$ differing from $H$ only at $o\big(\binom{n}{2}p\big)$ edges. If $H$ is from $\RGG$ already, the adversary does not change edges. We record this in the following corollary.

\begin{corollary} Suppose that $\epsilon\in (0,1)$ is any fixed constant. Whenever $d\ge K(\epsilon)\max(np (\log n)^{4}, (\log n)^5)$ the robust testing problem \cref{problem:robusttesting} is information-theoretically impossible with $\epsilon \binom{n}{2}p$ corruptions. If, furthermore, 
$d =  o((np)^3(\log 1/p)^3),$ the problem is information-theoretically impossible even against polynomial-time adversaries.
\end{corollary}

It turns out that when $d\lll np,$ there exists a polynomial time algorithm solving \cref{problem:robusttesting} with high-probability even against computationally unbounded adversaries. As already mentioned, it is unlikely that a subgraph-counting algorithm solves the robust testing problem. Instead, we require a more sophisticated sum-of-squares spectral refutation algorithm which uses the true embedding as a witness. To the best of our knowledge, \cite{mao2024informationtheoretic} is the only other work which uses an embedding-based approach for testing for random geometric graphs, but their work is in low-dimension and the algorithm is not computationally efficient. 

\begin{theorem}
\label{thm:robustestingsos}
There is an absolute constant $K_r>0$ such that when $d\le K_r\min(np\log(1/p), \frac{n^2p^2\log(1/p)}{\log^{2}n})$ and $d = \omega(\log 1/p),$ there exists a polynomial-time algorithm solving \cref{problem:robusttesting} with high probability for all $\epsilon\le K_r.$
\end{theorem}
We overview the algorithm in \cref{sec:nontechnicalrobust} and give full detail in \cref{sec:robusttesting}. It is a constant degree SoS program for spectral refutation of a small operator norm of the centered adjacency matrix. The key insight, implicit in \cite{li2023spectral}, is that the latent vectors witness a large second eigenvalue for a random geometric graph when $d\lll np.$ Coincidentally, this is the complementary regime from the one where \cref{thm:algorithmiccoupling} holds.

\begin{remark}
\label{rmk:couplingimpossibilityfromspectral}
\cref{thm:robustestingsos} implies that when $d \le K_r\min(np\log(1/p), n^2p^2\log(1/p)\log^{-2}n),$ there does not exist a coupling between $H_-\sim \mathsf{G}(n,p(1-o(1)),G\sim \RGGsphere, H_+\sim \mathsf{G}(n,p(1+o(1))$
such that $H_-\subseteq G\subseteq H_+$ w.h.p. Otherwise, one can use this coupling as an adversary for \cref{problem:robusttesting} even against computationally-unbounded testers, but \cref{thm:robustestingsos} shows that such an adversary cannot exist. Hence, the dimension requirement $d = \tilde{\Omega}(np)$ in \cref{thm:algorithmiccoupling,thm:stochasticcoupling} is tight (up to $\polylog$ factors).
\end{remark}

\subsection{Enumeration}  
 A classic result in algebraic geometry due to Warren \cite{warren1968lowerbounds} nearly immediately implies that the number of geometric graphs in $d$ dimensions is at most $(Cn/d)^{dn}$ for some explicit constant $C$ when $d = o(n)$ \cite{MCDIARMID2011627,SAUERMANN2021107593}. The main idea is that the edges of a geometric graph correspond to the sign patterns of $\binom{n}{2}$ polynomials over the $dn$ variables $\{(v_i)_j\}_{1\le i \le n, 1\le j \le d}.$ The variables are the coordinates of the latent vectors and the polynomials are the squared pairwise $\ell_2$ distance minus the radius squared.

Proving lower bounds, however, turns out to be a much more challenging task. In \cite{MCDIARMID2011627}, it is shown that in 2 dimensions the number of geometric graphs (known also as unit disk graphs in this case) is $n^{(2+o(1))n}.$ Using more sophisticated tools from algebraic geometry and topology, \cite{SAUERMANN2021107593} shows that the number of 
geometric graphs  in $d$ dimensions is $n^{(d+o(1))n}$ when $d = o(n).$

We recover this result up to a polylogarithmic factor in the exponent. Namely, using \cref{thm:stochasticcoupling}, we demonstrate that $\RGGsphere$ is a distribution over at least $\exp(\tilde{\Omega}(n^2p))$ graphs when $d= \max(np\log^4 n, \log^5 n)$ which yields the following statement.

\begin{theorem}
\label{thm:enumeration}
The number of graphs which can be realized as intersection graphs of unit spheres in dimension $d$ is at least $\exp(nd\log^{-7}n).$     
\end{theorem}

Admittedly, our proof loses a non-trivial polylogarithmic factor. Yet, one advantage of our proof is its simplicity -- it is based on a natural probabilistic construction, spherical random geometric graphs. By comparison, the techniques from algebraic geometry and algebraic topology in \cite{SAUERMANN2021107593} are significantly more involved, although certainly more systematic and sharper.

The proof idea in \cref{thm:enumeration} is rather simple: if $\mathcal{M}$ is the support of $\RGGsphere,$ a counting argument bounds the number of graphs that are edge-subgraphs of some graph $H$ in $\mathcal{M}$ and differ from $H$ in $o\big(\binom{n}{2}p\big)$ edges. Combined with \cref{thm:stochasticcoupling}, this immediately gives an upper bound on the entropy of $G(n,p(1-o(1)))$ in terms of $|\mathcal{M}|.$
As the latter is $\exp(\tilde{\Omega}(n^2p)),$ we get a lower bound on $|\mathcal{M}|.$

\subsection{Entropic and Spectral Interpretations of \texorpdfstring{$d = \tilde{\Theta}(np)$}{d is np}}
\label{sec:intepretationspectralentropic}
The reader might have observed the omnipresence of the dimension $d =\tilde{\Theta}(np)$ in the current work. This is a polynomially smaller dimension than the one required for $\TV(\RGGsphere, \ergraph) = o(1).$ For the latter to hold, $d= \tilde{\Omega}(n^3p^3)$ is needed as demonstrated by the signed triangle count test \cite{Bubeck14RGG,Liu2022STOC} which is also conjectured to be optimal \cite{Brennan19PhaseTransition,Liu2022STOC,bangachev2024fourier}. While $d = \tilde{\Theta}(n^3p^3)$ is still rather mysterious, the critical dimension $d = \tilde{\Theta}(np)$ in \cref{thm:algorithmiccoupling,thm:stochasticcoupling} has two natural interpretations: 
\begin{enumerate}
    \item \textbf{Entropic Threshold:} The maximum entropy of a graph distribution with expected density $p$ is $\binom{n}{2}H(p) = \tilde{\Theta}(n^2p)$ and this is achieved by $\ergraph.$ On the other hand, a simple heuristic\footnote{$\dsphere$ has an $\epsilon$-net of size $\exp(\tilde{O}(d\log(1/\epsilon)))$, 
    discretizing the sphere. As there are $n$ latent vectors $V_1, V_2, \ldots, V_n,$ each effectively from the $\exp(\tilde{O}(d))$-discretization, and the latent vectors uniquely determine the geometric graph, the $\RGGsphere$ distribution effectively has support size 
    $\exp(\tilde{O}(nd)).$ Hence, its entropy is at most $\tilde{O}(nd).$} (made rigorous in \cref{sec:enumeration}) suggests that the effective support size of $\RGGsphere$ is $\exp(\tilde{O}(nd))$ and, hence, its entropy is $\tilde{O}(nd).$ 
    To match (up to lower order terms) the $\tilde{\Theta}(n^2p)$ entropy of $\ergraph,$ one therefore needs $d = \tilde{\Omega}(np).$ In the regime $d\ggg np,$ we indeed use \cref{thm:stochasticcoupling} to show that the support size of $\RGGsphere$ is $\exp(\tilde{\Omega}(n^2p)).$ On the other hand, when $d\lll np,$ the entropy is 
    $\tilde{O}(nd).$ As in \cref{rmk:supportsizeofrgg}, this implies optimality of our \cref{thm:stochasticcoupling}, an approximate stochastic dominance of the given form cannot exist when $d\lll np.$ 
    In \cite{mao2024impossibilitylatentinnerproduct}, the authors use a more sophisticated rate-distortion-theoretic version of this entropy argument to 
    show that the when $d\ggg np,$ the latent vectors have a much larger entropy than the adjacency matrix. This yields
    the tight lower bound that when $d\ggg np,$ one cannot do a non-trivial recovery of the latent vectors from the adjacency matrix.
    \item \textbf{Spectral Threshold:}
    Somewhat miraculously, as soon as $d \lll np,$ there is a spectral separation between $\ergraph$ and $\RGGsphere$: 
    The second largest eigenvalue of $\ergraph$ is $\tilde{\Theta}(\sqrt{np})$ whereas the second largest eigenvalue of $\RGGsphere$ is 
    $\tilde{\Theta}({{np}/{\sqrt{d}}})\ggg \tilde{\Theta}(\sqrt{np})$
    \cite{Liu22Expander,li2023spectral}.
    This spectral separation allows for efficient spectral and SDP algorithms, as in \cite{li2023spectral} (for reconstructing the latent vectors from the adjacency matrix) and \cref{sec:robusttesting} (for robust testing). 
\end{enumerate}
We leave the open-ended question of explaining why the entropic and spectral thresholds coincide as an exciting direction for future work. 
While we are not aware of any reasons for why this is the case, it is fundamentally the reason why problems like latent vector embedding and robust testing for random geometric graphs do \emph{not} have statistical-computational gaps. Entropy roughly controls the size of the support and, hence, the success of (inefficient) exhaustive search algorithms. The spectrum of a matrix, on the other hand, is one of the most fundamental linear-algebraic primitives and is used throughout theoretical computer science for the design of computationally efficient algorithms.

The fact that simple spectral and SDP approaches match brute-force algorithms for random geometric graphs is not a given,  
and indeed this is not the case for many other latent space graph distributions. 
Most famously, for the planted clique distribution 
\cite{jerrum92plantedclique,KUCERA1995193,alon98plantedclique} an exponential gap arises with strong supporting evidence in the Sum-of-Squares hierarchy  \cite{barak19soslowerbound}. 
For a random  geometric graphs example, the planted dense cycle model also exhibits an information-computation gap \cite{mao23detectionrecoverygap,mao2024informationtheoretic}. On the other hand, there are also many models without an information-computation gap such as exact recovery in the 2-community symmetric stochastic block model \cite{mossel15consistency,abbe2014exactrecovery}. In general, it is poorly understood when an information-computation gap for testing and recovery of latent space structure arises. In \cite{ElAlaouiAhmed2021Otct} the authors give a general condition under which a gap does not exist, but their model is qualitatively different and does not capture our setting. In it, one only observes a predefined subset of the edges. The absence of a gap depends on the structure of this subset.

\subsection{Challenges and Comparison to Prior Work}
\label{sec:priorwork}


In \cite{Liu2022STOC}, the authors prove \cref{thm:stochasticcoupling} under the weaker condition $d\ge n^2p^2\log^4n.$ In their construction, they sample $V_1, V_2, \ldots, V_n$ sequentially. A caricature of their argument is that with high probability, when sampling $V_1, V_2, \ldots, V_{k-1}$ the intersection of any $i\approx pk$ $p$-caps\footnote{For $V\in \dsphere,$ the associated $p$-cap is the set $\{U \in \dsphere\; : \langle U, V\rangle\ge \tau^p_d\}$  and the $p$-anti-cap is its complement.} and $k-i-1\approx k(1-p)$ $p$-anti-caps corresponding
to $V_1, V_2, \ldots, V_{k-1}$ is sufficiently close to $p^{i}(1-p)^{k-i-1}.$ 
This implies that the edges from $k$ to $[k-1]$ are roughly distributed as independent $\Bernoulli(p)$ variables. They make this statement precise by quantifying ``sufficiently close'' via a concentration argument based on transportation inequalities. 

While our proof of \cref{thm:stochasticcoupling} is entirely different, it suggests what might be suboptimal in the approach of \cite{Liu2022STOC}. For typical $V_1,\dots, V_{k-1}$, it turns out that the volume of \emph{most} intersections of $i\approx pk$ caps and $k-i-1\approx k(1-p)$ anti-caps concentrates much better than the worst-case over choice of the $i$ caps and $k-i-1$ anti-caps due to a square-root CLT-type cancellation phenomenon. We use this idea in our proof to get $d\ge \max(np, \log n)$; if we instead consider \emph{all} intersections of caps and anti-caps, we also obtain a result for $d\ge n^2p^2\polylog n.$ While the square-root cancellation comes naturally in our martingale-based analysis, it is unclear whether it can be incorporated in the transportation-based approach of \cite{Liu2022STOC}. Beyond that, an advantage of our approach is that \cref{thm:algorithmiccoupling} gives an \emph{efficient algorithm} producing the coupling. The result of \cite{Liu2022STOC} seems difficult to make algorithmically efficient as it requires keeping track of an intersection of caps and anti-caps and being able to sample and estimate the volume of this intersection.    

The work
\cite{bangachev2024fourier} also shows that in certain regimes random geometric graphs can be represented as $\ergraph$ with some modifications. 
In fact, in their representation, just as in \cref{thm:algorithmiccoupling} the non-\ERspace pairs $(i,j)$ are exactly the ones for which $\langle V_i, V_j\rangle$ are close to $\tau^p_d$ (called \emph{fragile pairs} in \cite{bangachev2024fourier}). While some of our ideas are related to and motivated by \cite{bangachev2024fourier}, there is a fundamental barrier to their proof technique: It crucially relies on performing the Gram-Schmidt operation on the latent vectors\footnote{To be precise, Gram-Schmidt is performed on isotropic Gaussian vectors $Z_1, Z_2, \ldots, Z_n.$ 
This, however, makes little difference to the arguments in both \cite{bangachev2024fourier} and the current work, see \cref{rmk:sphericaltogaussian}.} $V_1, V_2, \ldots, V_n$, requiring $d\ge n.$ This makes it inapplicable to the regime $np\lll d\lll n.$ We circumvent this barrier in the current work by replacing Gram-Schmidt with a different transformation that requires only near-orthogonality of the latent vectors, which holds as soon as $\polylog(n)\lll d.$ 
\subsection{Organization}
In \cref{sec:preliminaries} we introduce some basic preliminaries and notation. In \cref{sec:nontechnical}, we overview our two main algorithmic results: the coupling algorithm in \cref{thm:algorithmiccoupling} and the spectral refutation algorithm in \cref{thm:robustestingsos}. Additionally, in \cref{sec:recursiverepresentation} we address the representation mentioned in \cref{rmk:representation}.
We give the full details of \cref{thm:algorithmiccoupling} and \cref{thm:robustestingsos} in \cref{sec:coupling} and \cref{sec:robusttesting}. In \cref{sec:enumeration} we prove \cref{thm:enumeration}. In \cref{sec:sharpthresholds} we prove \cref{thm:sharpthresholds} as well as the fact that $\RGGsphere$ does not satisfy the FKG property. We end with several remarks and open questions regarding our results in \cref{sec:discussion}.

%% file: Preliminaries.tex
\section{Preliminaries}
\label{sec:preliminaries}
\paragraph{Graphs and Matrices.} We denote by $\ergraph$ the \ERspace distribution over undirected graphs on $n$ vertices, in which each edge appears independently with probability $p.$ For a graph $G$ on $n$ vertices, its adjacency matrix $A$ is an $n\times n$ matrix for which $A_{ij} = \indicator[(i,j)\text{ is an edge in }G].$ The $p$-centered adjacency matrix of $G$ is an $n\times n$ matrix for which $A_{ii} = 0$ for all $i,$ $A_{ij} = 1-p$ if 
$(i,j)$ is an edge in $G, $ $A_{i,j} = -p$ if $(i,j)$ is not an edge. When we only say ``centered'' adjacency matrix of $G$ where
$G\sim \mathcal{G},$ the value of $p$ is implied to be the expected marginal edge probability according to 
$\mathcal{G}.$ The largest in absolute value eigenvalue of a matrix $M$ is denoted by
$\lambda_{\max}(M).$ 

We will need the following well-known fact about the $\ergraph$ distribution.

\begin{theorem}[\cite{georges19verysparse,georges20spectralsparse,alt21extremal}]
\label{thm:secondeigenvalueofer}
Let $A$ be the centered adjacency matrix of $G \sim \ergraph.$ Then, for some absolute constant $C,$ with high probability
$
  |\lambda_{\max}(A)|\le
  C\max(\sqrt{np}, \sqrt{\log n}).$ 
\end{theorem}

\paragraph{Spherical Coordinates.} Denote by $\sphericaloned$ the marginal distribution of a coordinate from the unit sphere. That is, $\sphericaloned$ is the distribution of $\langle V, e_1\rangle$ for $V\sim\unif(\dsphere).$
One can define the threshold in \cref{def:sphericalrgg} by the equation $p = \prob_{X\sim\sphericaloned}[X\ge \tau^p_d].$ We will use the following simple facts. The intuition behind all of them is that for large enough $d,$ the distribution $\sphericaloned$ is very close to $\mathcal{N}(0,\frac{1}{d}).$  

\begin{lemma}
    \label{lemma:sphericalcdf} Suppose that $p\le 1/2$ and $d= \omega( \log(1/p)).$ Then, for some absolute constants $C_A>1, C_B>1, C_D>1,$ and $C_i>1$ for $i\in \mathbb{N},$ depending only on $i,$ but not on $d$ or $p$: 
    \begin{enumerate}
    \item [A)]  $\min\Big(\frac{1}{C_A}(\frac{1}{2}-p)\frac{\sqrt{\log(1/p)}}{\sqrt{d}}),\frac{1}{2}\Big)\le\tau^p_d\le C_A\frac{\sqrt{\log (1/p)}}{\sqrt{d}}.$
    \item [B)] $\frac{\sqrt{\log 1/p}}{C_B\sqrt{d}}\le \expect_{X\sim \sphericaloned}[X\; | \; X\ge \tau^p_d].$ 
    \item [C)] $\expect_{X\sim \sphericaloned}[X^k\; |\; X\ge \tau^p_d]\le C_k\frac{\log(1/p)^{k/2}}{d^{k/2}}.$
    \item [D)] Suppose that $ \Delta\ge 0$ and $\Delta \le \frac{1}{C_D}\frac{(\log 1/p)^{-1/2}}{\sqrt{d}}.$ Then, 
    \begin{align*}
    \prob_{X\sim\sphericaloned}\Big[X\in [\tau^p_d-\Delta, \tau^p_d + \Delta]\Big]\le C_D
    {\Delta p}{\sqrt{d}}\sqrt{\log 1/p}.
\end{align*}
    \end{enumerate}
\end{lemma}
The proofs are simple and are deferred to \cref{appendix:proofofspherical}. A) is from \cite{Bubeck14RGG}. D) appears in \cite{bangachev2024fourier} with slightly weaker dependence on $d$ and $p$, so for completeness we also give a proof.

\paragraph{Martingales.} In the analysis of our coupling algorithm for \cref{thm:algorithmiccoupling} we will use the well-known inequality due to Freedman.
\begin{theorem}[\cite{freedman75mgbernsetin}]
\label{thm:mgbernstein}
Consider a sequence of martingale differences $(\xi_i)_{i=1}^r,$ with respect to the filtration $\{\mathcal{F}_i\}_{i = 0}^r.$  That is, $\xi_i$ is $\cF_i$-measurable, bounded, and $\expect[\xi_i|\cF_{i-1}] = 0.$ 
Suppose that
\begin{enumerate}
    \item $\sum_{i = 1}^r \expect[\xi^2_i|\cF_{i-1}] \le L$ for some deterministic $L,$ and 
    \item $|\xi_i|\le a$ for all $i$ for some deterministic value of $a.$ 
\end{enumerate}
Then,
$\displaystyle
\prob\Big[\Big|\sum_{i= 1}^r \xi_i\Big|\ge z\Big]\le 
2\exp\Big( - \frac{z^2/2}{L + az/3}\Big).
$
\end{theorem}

\paragraph{Sharp Thresholds.} Let $S$ be a finite set and consider a family of distributions $(\mathcal{D}_p)_{p \in [0,1]}$ over $\{0,1\}^S$ (we identify the power set of $S$ with the set of indicator vectors $\{0,1\}^S$) parametrized by the common marginal coordinate density $p,$ for example $\ergraph$ or $\RGGsphere.$ Let $\mathcal{P}\subseteq \{0,1\}^S$ be a given property (such as connectivity or Hamiltonicity of a graph). We say that $\mathcal{P}$ is monotone increasing if for any $A\in \mathcal{P}$ and $B \subseteq \{0,1\}^S$ such that $A\subseteq B,$ it is also true that $B\in \mathcal{P}.$ If $\mathcal{P}$ is a non-trivial monotone increasing property (that is, other than ``no sets satisfy it'' or ``all sets satisfy it''), then $\emptyset \not \in \mathcal{P}$ and $S\in \mathcal{P}.$  Hence, $\prob_{X\sim\mathcal{D}_0}[X\in \mathcal{P}] = 0, \prob_{X\sim\mathcal{D}_1}[X\in \mathcal{P}] = 1.$

For natural distributions such as $\ergraph$ and $\RGGsphere,$ the function $f: [0,1]\longrightarrow[0,1]$ defined by $f(p) = \prob_{X\sim\mathcal{D}_p}[X\in \mathcal{P}]$ is continuous and increasing. As $f(0) = 0, f(1) = 1,$ there exists a \emph{critical probability} $p_c$ such that $\prob_{X\sim\mathcal{D}_{p_c}}[X\in \mathcal{P}] = 1/2,$ i.e. $p_c = f^{-1}(1/2).$  For a small constant $\epsilon>0,$ the \emph{critical window} is denoted $w_\epsilon$ and equals $f^{-1}(1-\epsilon) - f^{-1}(\epsilon).$ Clearly, $w_\epsilon>0.$

When $w_\epsilon = o(p_c)$ for a given property $\mathcal{P}$ for all fixed $\epsilon>0,$ we say that $\mathcal{P}$ has a \emph{multiplicative sharp threshold} with respect to $(\mathcal{D}_p)_{p \in [0,1]}.$ This is the main setting we will be interested in. Common in the literature is also the weaker \emph{additive sharp threshold} notion which only requires $w_\epsilon = o_n(1).$ When we say \emph{sharp thresholds} in the current paper, we always mean the stronger multiplicative variant unless otherwise stated. 
All of the above can be analogously defined for monotone decreasing properties. See, for example, \cite{ODonellBoolean,grimmett2006random} for a more thorough discussion on sharp thresholds.

%% file: NonTechnical.tex
\section{Overview of Main Ideas}
\label{sec:nontechnical}

\subsection{The Coupling Algorithm}
\label{sec:couplingnontecnhical}

\paragraph{Motivating Our Construction.} Suppose that we are given the adjacency matrix $(H_{ji})_{1\le i < j \le n}\iidsim\Bernoulli(p)$ of $H\sim\ergraph.$ Our goal is to couple $H$ with $n$ vectors $V_1, V_2, \ldots, V_n$ on the sphere such that:
\begin{enumerate}
    \item $V_1, V_2, \ldots, V_n\iidsim\unif(\dsphere).$
    \item Conditioned on $(H_{ji})_{1\le i < j \le n},$ for most pairs $(j,i),$ it is the case that $\indicator[\langle V_i, V_j\rangle\ge \tau^p_d] = H_{ji}.$  
\end{enumerate}

While property 1 is easy to achieve on its own, property 2 is much less trivial. One needs to approximately embed the adjacency matrix of $H$ as a geometric graph on the sphere. Even if the adjacency matrix already corresponds to a sample from $\RGGsphere,$ this step is non-trivial and requires some embedding algorithm, for example the spectral approach of \cite{Liu2022STOC}. And even if one succeeds in this embedding step, it is unclear how one could enforce the distributional assumption in property 1.

Our approach is to build up a coupling incrementally, starting with vectors satisfying property 1 and maintaining this as an invariant. That is, the first step is to sample $n$ vectors $V_1^{(0)}, V_2^{(0)}, \ldots, V_n^{(0)}\iidsim\unif(\dsphere).$ We then incrementally transform them so that 1) they remain uniform on the sphere and 2) they produce  a geometric graph that is close to the input $H$.

We start with the edge $(2,1).$ If $H_{21} = \indicator[\langle V^{(0)}_1,V^{(0)}_2\rangle\ge \tau^p_d],$ we do not need to do anything since property 2 is already satisfied for this edge. Otherwise, we transform $V^{(0)}_2$ into $V^{(1)}_2$ such that $\indicator[\langle V^{(0)}_1,V^{(1)}_2\rangle\ge \tau^p_d] = 1- \indicator[\langle V^{(0)}_1,V^{(0)}_2\rangle\ge \tau^p_d]$, i.e., presence of edge $(2,1)$ in our RGG is flipped. One way to do this transformation is as follows: we keep unchanged 
the direction of $V^{(0)}_2$ in the space orthogonal to $V^{(0)}_1$
 and flip $V^{(0)}_2$ in the direction of $V^{(0)}_1$ via a measure-preserving (with respect to the $\unif(\dsphere)$ distribution) map so that the inner product is on the correct side of $\tau^p_d.$ We call this procedure a \emph{flip orientation} of $V^{(0)}_2$ along $V^{(0)}_1$ so that the edge between them is set to $H_{21}$ and we write $V^{(1)}_2 = \rho^{H_{21}}_{V^{(0)}_1}(V^{(0)}_2).$ In particular, in this notation $\rho^b_{a}(z) = z$ whenever $b = \indicator[\langle a,z\rangle\ge\tau^p_d].$

Having coupled $H_{21}$ with $V^{(0)}_{1}, V^{(1)}_2,$ we proceed to flip $V^{(0)}_3$ first around $V^{(0)}_1$ and then around $V^{(1)}_2$ and so on. In short, the algorithm we use to prove \cref{thm:algorithmiccoupling} is as follows.

\begin{algorithms}
\label{eq:couplingalgorithm}
First, sample $V_1^{(0)}, V_2^{(0)}, \ldots, V_n^{(0)}\iidsim\unif(\dsphere).$ Then, flip sequentially:
\begin{equation*}
    \begin{split}
        &V_1^{(0)},\quad \Vfin{1}\coloneqq V_1^{(0)} \\
        &V_2^{(0)}\stackrel{\rho^{H_{21}}_{\Vfin{1}}}{\longrightarrow}
        V_2^{(1)},\quad \Vfin{2}\coloneqq V_2^{(1)}\\
        & V_3^{(0)}\stackrel{\rho^{H_{31}}_{\Vfin{1}}}{\longrightarrow}
        V_3^{(1)}
        \stackrel{\rho^{H_{32}}_{\Vfin{2}}}{\longrightarrow}
        V_3^{(2)},\quad \Vfin{3}\coloneqq V_3^{(2)} \\
        &\,\,\,\vdots\\
        & V_n^{(0)} \stackrel{\rho^{H_{n1}}_{\Vfin{1}}}{\longrightarrow}
        V_n^{(1)}
    \stackrel{\rho^{H_{n2}}_{\Vfin{2}}}{\longrightarrow}
        V_n^{(2)}\cdots
    \stackrel{\rho^{H_{n,n-1}}_{\Vfin{n-1}}}{\longrightarrow}
        V_n^{(n-1)},\quad \Vfin{n}\coloneqq V_n^{(n-1)}.
    \end{split}
\end{equation*}
The final output vectors are denoted by $(\Vfin{i})_{i = 1}^n$ where $\Vfin{i}\coloneqq V^{(i-1)}_{i}.$ Importantly, observe that \emph{flips occur with respect to the final vectors.} In general, the flip orientation transformation is succinctly expressed as $$V^{(i)}_j = \rho^{H_{ji}}_{\Vfin i}(V^{(i-1)}_j),$$ 
 meaning that we flip $V^{(i-1)}_j$ in the direction of $\Vfin{i}$ so that 
$\indicator[\langle V^{(i)}_j,\Vfin{i}\rangle \ge \tau^p_d] = H_{ji}.$ 

\end{algorithms}

\paragraph{Analysis Sketch.} We need to show that the two properties claimed in \cref{thm:algorithmiccoupling} hold.

Independence follows directly from the ``measure-reserving'' property of $\rho^b_a(\cdot),$ which we formalize in \cref{sec:fliporientation}. For technical reasons, the following slightly stronger independence statement will be useful.
\begin{lemma}
\label{lem:flippreservesiid}
For any choice of $s\in \{1,2,\ldots, n\}$ and $i_s\in \{0,1,2\ldots, s-1\},$
$$
\Vfin{1}, \Vfin{2}, \ldots, \Vfin{s-1},V_s^{(i_s)}\iidsim\unif(\dsphere).
$$
\end{lemma}
The main idea behind the proof is that before any operation is performed on $V_{s}^{(0)},$ it is uniform on the sphere and independent from $\Vfin{1}, \Vfin{2}, \ldots, \Vfin{s-1}.$ Since the transformation $V_{s}^{(0)}\longrightarrow V_{s}^{(1)}$ is measure-preserving for any choice of $\Vfin{1}, \Vfin{2}, \ldots, \Vfin{s-1},$ the distribution of $V_{s}^{(1)}$ remains uniform on the sphere and independent from $\Vfin{1}, \Vfin{2}, \ldots, \Vfin{s-1}.$  We proceed by induction.  

We also need to prove that with high probability, all pairs $i,j$ such that\linebreak $|\langle \Vfin{i}, \Vfin{j}\rangle - \tau_p^d|\ge \frac{c\max(\sqrt{np}, \sqrt{\log n})(\log n)^{3/2}}{d}$ satisfy 
    $\indicator[\langle \Vfin{i}, \Vfin{j}\rangle\ge \tau^p_d] = H_{ji}.$ By construction of the flip orientation map $\rho$,  $\indicator[\langle \Vfin{i}, V^{(i)}_j\rangle\ge \tau^p_d] = H_{ji}$, because $ V^{(i)}_j$ has just been oriented relative to $\Vfin{i}$ to achieve this. 
    We will show that the final $\Vfin j = V_j^{(j-1)}$ has not changed too significantly from $V^{(i)}_j$ in the $\Vfin i$ direction:
    
\begin{lemma} 
\label{lem:flippreservesinnerproducts} For any $1\le i < j\le n,$
\begin{equation}
    \prob\bigg[
    \Big|\langle \Vfin{i},\Vfin{j}\rangle - \langle \Vfin{i},V^{(i)}_j\rangle\Big|\ge \frac{c\max(\sqrt{np}, \sqrt{\log n})(\log n)^{3/2}}{d}
    \bigg]\le \frac{1}{n^{100}}.
\end{equation}
    \end{lemma}

To give intuition about the quantity $\widetilde{\Theta}({\sqrt{np}}/{d})$, we write $\langle \Vfin{i},\Vfin{j}\rangle - \langle \Vfin{i},V^{(i)}_j\rangle$ as a telescoping sum, recalling that $\Vfin{j} = V^{(j-1)}_j:$
\begin{equation}
\label{eq:telescopingsum}
    \begin{split}
    &\langle \Vfin{i},\Vfin{j}\rangle - \langle \Vfin{i},V^{(i)}_j\rangle = 
\sum_{u = i+1}^{j-1}
\langle \Vfin{i}, 
V^{(u)}_j- 
V^{(u-1)}_j
\rangle = 
\sum_{u = i+1}^{j-1}
\langle \Vfin{i}, 
\rho^{H_{ju}}_{\Vfin{u}}(V^{(u-1)}_j)- 
V^{(u-1)}_j 
\rangle.
    \end{split}
\end{equation}
Each of the terms is equal to $0$ with probability $1-p.$ That is because $\rho^{H_{ju}}_{\Vfin{u}}(V^{(u-1)}_j)- 
V^{(u-1)}_j = 0$ whenever $\indicator[\langle \Vfin{u},V^{(u-1)}_j\rangle \ge \tau^p_d] = H_{ji}$ and the two sides are independent $\Bernoulli(p)$ random variables (the left-hand side because $\Vfin{u},V^{(u-1)}_j$ are independent by \cref{lem:flippreservesiid}). Hence, we expect only $\widetilde{O}(np)$ non-zero terms. 

Whenever $\rho^{H_{ju}}_{\Vfin{u}}(V^{(u-1)}_j)- 
V^{(u-1)}_j$ is non-zero, its norm is on the order of $\widetilde{O}(1/\sqrt{d})$ with high probability. This is the case because $\Vfin{i}, 
V^{(u)}_j$ and $\Vfin{i}, 
V^{(u-1)}_j$ are pairwise independent vectors on the unit sphere by \cref{lem:flippreservesiid} and $V_j$ changes in direction $\Vfin{u}$ when $\rho^{H_{ju}}_{\Vfin{u}}(\cdot)$ is applied. Moreover, $\rho^{H_{ju}}_{\Vfin{u}}(V^{(u-1)}_j)- 
V^{(u-1)}_j$ is in the span of $\Vfin{u}, V^{(u-1)}_j.$ Both directions are independent of $\Vfin{i}$ and, hence, only have a $\widetilde{O}(1/\sqrt{d})$ inner product with $\Vfin{i}.$ Altogether, this means that  each $\langle \Vfin{i}, 
\rho^{H_{ju}}_{\Vfin{u}}(V^{(u-1)}_j)- 
V^{(u-1)}_j 
\rangle$ is of order $\widetilde{O}(1/d).$ As there are $\widetilde{O}(np)$ non-zero terms, the sum is of order $\widetilde{O}(np/d)$ with high probability.

However, we need to prove that the sum is of order $\widetilde{O}(\sqrt{np}/d).$ The key insight is that a CLT-type square-root cancellation occurs in the sum \cref{eq:telescopingsum}. We show this via the classic martingale inequality \cref{thm:mgbernstein}. The technical challenge is that \cref{eq:telescopingsum} is \emph{not actually a sum of martingale differences} for the following ``martingale violation'' reasons:
\begin{enumerate}
\item[\mylabel{itm:v1}{\textbf{(V1)}}] 
Whenever $\rho^{H_{ju}}_{\Vfin{u}}(V^{(u-1)}_j)- 
V^{(u-1)}_j\neq 0,$ it has a non-trivial component in the direction of $V^{(u-1)}_j.$ The reason is that while the map $V\mapsto\rho^{H_{ju}}_{\Vfin{u}}(V)$  preserves the direction of $V$ on the space orthogonal to $\Vfin{u},$ it does not preserve the norm of $V$ in this orthogonal space. This is because $\rho^{H_{ju}}_{\Vfin{u}}(V)$ must also have unit norm and, thus, a rescaling occurs. As the difference $\rho^{H_{ju}}_{\Vfin{u}}(V^{(u-1)}_j)- 
V^{(u-1)}_j$ is not independent of $V^{(u-1)}_j,$ one can easily show that the martingale property is violated.
\item[\mylabel{itm:v2}{\textbf{(V2)}}] Whether $\rho^{H_{ju}}_{\Vfin{u}}(V^{(u-1)}_j)- 
V^{(u-1)}_j = 0$ holds depends on the projection of $V_{j}^{(u-1)}$ on $\Vfin{u}.$ Conditioned on $H_{ji}, H_{j,i+1}, \ldots, H_{j,u-1}$ (which need to appear in the underlying filtration), $V_{j}^{(u-1)}$ and $\Vfin{i}$ are not independent. Hence, the projections of $\Vfin{u}$ on $V_{j}^{(u-1)}, \Vfin{i}$ are not independent. So,\linebreak $\langle \Vfin{i}, 
\rho^{H_{ju}}_{\Vfin{u}}(V^{(u-1)}_j)- 
V^{(u-1)}_j 
\rangle$ does not have a zero conditional expectation.
 \end{enumerate}

We deal with these issues by manually subtracting off the component of $\rho^{H_{ju}}_{\Vfin{u}}(V^{(u-1)}_j)- 
V^{(u-1)}_j$  in the direction of $V^{(u-1)}_j$ and the component of $V_{j}^{(u-1)}$ in the $\Vfin{i}$ direction. We show that these lead to lower-order contributions and what remains is a true sum of martingale differences.

\begin{remark} On a high-level ``flipping $V^{(i-1)}_j$ in the direction of $\Vfin{i}$'' is analogous to the Gram-Schmidt procedure of \cite{bangachev2024fourier} which aims to isolate the component of $V_j$ in the direction of $V_i.$ In \cite{bangachev2024fourier}, this is achieved by forming an orthonormal basis $e_1, e_2, \ldots, e_n$ such that $V_i\approx e_i.$ In fact, one can modify their argument and  derive the equivalent statement of \cref{thm:algorithmiccoupling} when $d\gg n$ (by flipping measure-preservingly the $Z_{ji}$ values around $\phi^p_d$ in their notation). Nevertheless, when $d< n,$ one cannot form an orthonormal basis. A ``nearly orthonormal'' set is not sufficient for the arguments in \cite{bangachev2024fourier} as they heavily rely on the fact that the components of a Gaussian vector in orthogonal directions are independent. Hence, we need a different approach for the regime $np\ll d \ll n$ that is central to our applications.
\end{remark}


\subsection{The Robust Testing Algorithm}
\label{sec:nontechnicalrobust}
\paragraph{Motivation.} The motivation behind our testing algorithm comes from the interpretation of $d = \widetilde{\Theta}(np)$ as a spectral threshold. Specifically, when $d\ll np,$ the centered adjacency matrix of $\RGGsphere$ typically has $d$ large eigenvalues of order $\widetilde{\Theta}({np}/{\sqrt{d}}).$ This is larger than the typical second eigenvalue of a sample from $\ergraph$ which is only $\sqrt{np}.$ Something more turns out to be true: the eigenvectors corresponding to these $d$ eigenvalues are very close to the latent vectors $V_1, V_2, \ldots, V_n,$ a key insight in \cite{li2023spectral}. Formally, the $d$ rows of the matrix $V = (V_1, V_2, \ldots, V_n)$ are close to the $d$ eigenvectors with large eigenvalues. In \cite{li2023spectral}, the authors use this fact to estimate the true Gram matrix $VV^T$ from the adjacency matrix. Our end goal is different -- to construct a spectral refutation algorithm. Formally, we give a certificate robust to $\epsilon \binom{n}{2}p$ edge corruptions that the $p$-centered adjacency matrix of $\RGGsphere$ has a large operator norm.

\paragraph{The case of no corruptions.} Suppose for now that we wanted to use the spectrum to test between $\ergraph$ and $\RGGsphere$ in the case of no corruptions when $d\ll np.$ 
\footnote{Of course, this regime is  suboptimal in the case of no adversarial corruptions as the signed triangle count succeeds all the way up to $d = \widetilde{\Theta}(n^3p^3)$ \cite{Bubeck14RGG,Liu2022STOC}.}
For the uncorrupted centered adjacency matrix $A$ of a sample from $\RGGsphere$, we know that there exists some matrix $Y\in \mathbb{R}^{d\times n}$ such that $\|Y_{:,i}\|_2 = 1$ for each $i$ and
$\langle Y^TY, A\rangle$ is large: One such matrix is $Y = V,$ the matrix of latent vectors. As discussed above, $V$ is close to the top $d$ eigenvectors of $A$, each corresponding to eigenvalue $\lambda = \widetilde{\Theta}({np}/\sqrt{d}).$ Furthermore, $\|V\|^2_F = n.$ Hence, the objective value $\langle V^TV, A\rangle$ is on the order of $\lambda\times \frac{\|V\|^2_F}{{d}} \times d = \widetilde{\Theta}({n^2p}/\sqrt{d}).$ To illustrate the main ideas, assume that it is exactly $n^2p/\sqrt{d}.$

Consider now the centered adjacency matrix $A$ of a $\ergraph$ sample.
Its largest eigenvalue is $\widetilde{O}(\sqrt{np}),$ so
$$
\langle  Y^TY,A\rangle \le 
\langle  Y^TY,\widetilde{O}(\sqrt{np})\times I\rangle \le 
\widetilde{O}(\sqrt{np})\times \trace(Y^TY) = 
\widetilde{O}(n\sqrt{np}).
$$
Again, for simplicity, assume that it is $n\sqrt{np}.$ Note that $n\sqrt{np}\ll n^2p/\sqrt{d}$ when $d\ll np.$

We will therefore search (via a suitable SoS program) for a matrix $Y\in \mathbb{R}^{d\times n}$ such that $\langle YY^T, A\rangle$ is larger than some threshold. If we succeed, then our graph ought to be a sample from $\RGGsphere.$

\paragraph{Handling corruptions.} Observe that the objective value $\langle Y^TY, A\rangle$ for $Y = V$ changes by very little when an adversary corrupts an edge $A_{ij}.$ The reason is that $(V^TV)_{ij} = \langle V_i, V_j\rangle,$ which with high probability is only on the order of $\widetilde{O}(1/\sqrt{d}).$ Hence, after $\epsilon\binom{n}{2}p$ corruptions, the objective value can change by at most $\widetilde{\Theta}(\epsilon n^2p/\sqrt{d}).$ Again, for simplicity, assume that the change is at most $\epsilon n^2p/\sqrt{d}.$ This, however, is less than the difference between $n^2p/\sqrt{d}$ (the uncorrupted $\RGG$ value) and $n\sqrt{np}$ (the uncorrupted $\ergraph$ value)! Hence, imposing that $|\langle Y_i, Y_j\rangle| = \widetilde{O}(1/\sqrt{d}),$ the SoS program will distinguish between the two graph models. 

To get the full statement of \cref{thm:robustestingsos}, we need to be more careful due to the log factors that we suppressed everywhere. Otherwise, we may end up with an algorithm that only works for $\epsilon = \log n^{-C}$ for some absolute constant $C>0.$ We deal with this by imposing a more fine-grained bound on the empirical distribution of $\{|\langle Y_i, Y_j\rangle|\}_{1\le i <j \le n}$ instead of $|\langle Y_i, Y_j\rangle| = \widetilde{O}(1/\sqrt{d}).$ 

\begin{remark} The idea of maximizing a quadratic form against the centered adjacency has been used for robust statistics in the stochastic block model, e.g. \cite{GuedonOlivier2016Cdis}. What is new in our work is the realization that the latent vectors witness a large operator norm in the case of random geometric graphs. To the best of our knowledge, novel is also the use of a constraint on the tails of the empirical distribution of the entries of a given matrix ($\{|\langle Y_i, Y_j\rangle|\}_{1\le i <j \le n}$). 
\end{remark}

\subsection{A Curious Representation: RGG as Recursively Planted \ER}
\label{sec:recursiverepresentation}
We now describe a curious representation result, which furthers the view of RGG as \ERspace with few planted edges arising in \cref{thm:stochasticcoupling} and \cite{bangachev2024fourier}. Most strikingly, when $d = n^{1+\epsilon}p$ for some absolute constant $\epsilon,$ we show that one can represent $\RGGsphere$ by starting with a sample from $\ergraph$ and performing constantly many rounds of planting independent density $1/2$ \ERspace graphs. The sizes of these planted \ERspace graphs decrease exponentially in $n.$ The entire ``geometry'' of the $\RGGsphere$ distribution is, thus, hidden in the locations of the planted \ERspace graphs.

Concretely, in the current \cref{thm:algorithmiccoupling}, we allow for some amount of ``defect edges'', i.e. the ones for which $|\langle V_i, V_j\rangle- \tau^p_d|\le \frac{c\max(\sqrt{np}, \sqrt{\log n})(\log n)^{3/2}}{d}.$ One may try to correct these edges after the termination of \cref{eq:couplingalgorithm}. Indeed, recall from \cref{sec:couplingnontecnhical} that one starts with arbitrary vectors $V_1^{(0)}, V_2^{(0)}, \ldots, V^{(0)}_{n}\iidsim \unif(\dsphere)$ and transforms them into $\Vfin{1}, \Vfin{2}, \ldots, \Vfin{n}.$ However, 
$\Vfin{1}, \Vfin{2}, \ldots, \Vfin{n}\iidsim \unif(\dsphere).$ Thus, one may try to apply the algorithm again and correct the defect edges. That is,  
sequentially apply
$V^{(i)}_{j} = \rho^{H_{ji}}_{\Vfin{i}}(\Vfin{j})$ for all $i,j$ such that
$|\langle \Vfin{i}, \Vfin{j}\rangle- \tau^p_d|\le \frac{c\max(\sqrt{np}, \sqrt{\log n})(\log n)^{3/2}}{d}.$

Intuitively, this should fix most of the disagreements between $H$ and $G_{ji}= \indicator[\langle V_i, V_j\rangle\ge \tau^p_d].$ The reason is that each $\Vfin{i}$ 
will only be changed $\widetilde{O}(p\sqrt{np/d})$ times as (with high probability there are at most $\widetilde{O}(p\sqrt{np/d})$ values of $j$ for which $|\langle \Vfin{i}, \Vfin{j}\rangle- \tau^p_d|\le \frac{c\max(\sqrt{np}, \sqrt{\log n})(\log n)^{3/2}}{d}$). Furthermore, one can easily show that the changes will be of smaller order than the ones \cref{sec:couplingnontecnhical} (see \cref{appendix:representation}). Altogether, running through the same martingale-based analysis, one should expect to fix all edges except the ones for which $\big|\langle \Vfin{i},\Vfin{j}\rangle- \tau^p_d\big|\lll\big(\max(\frac{1}{\sqrt{d}}\times(np/d)^{5/4}), \frac{1}{d}(np/d)^{1/2})\big).$ This will vastly reduce the number of defect edges. After several more iterations, one can hope to completely remove all defect edges. 

Of course, this should not be possible since $\ergraph$ and $\RGGsphere$ are \emph{far in total variation} in the regime $np\lll d\lll n^3p^3.$
The issue with the argument above is that the independence is lost since we have already used $H_{ji}$ once to
flip $V^{(i)}_{j} = \rho^{H_{ji}}_{\Vfin{i}}(V^{(i-1)}_j)$ in \cref{eq:couplingalgorithm}. Hence, 
the distributional \cref{lem:flippreservesiid} does not hold any more and
the resulting sample will not be a sample from $\RGGsphere.$

A simple way to resolve this issue is to draw fresh values of $H^1_{ji}$ for the second round of flippings. Then, draw fresh values for a third round and so on until no defect edges are left. This algorithm yields the following representation of $\RGGsphere.$ We give more detail in \cref{appendix:representation}.

\begin{theorem}[Description of Random Geometric Graphs as Recursively Planted \ER]
\label{obs:rggasrecursiveER}
Suppose that $p=\Omega(1/n)$ and $d = \kappa\times  np\log^6n$ for some $\kappa >1.$ Then, there exist some $T = O(\frac{\log d}{\log \kappa} + \frac{\log n}{\log d})$ and a set of 
random functions $F_1, F_2, \ldots, F_T,$ where each $F_i$ takes an $n$-vertex graph as an input and outputs a subset of $[n]\times [n]$
with the following property. Consider the algorithm:
\begin{enumerate}
    \item Draw $H^0\sim\ergraph$ and set $G^0 = H^0.$
    \item For $t = 1,2, \ldots, T:$
    \begin{enumerate}
        \item Draw a fresh $H^t\sim\mathsf{G}(n,1/2).$
        \item Compute the set of pairs of vertices $F_t(G^{t-1})\in [n]\times [n].$
        \item Define $G^t$ by $G^t|_{F_{t}(G^{t-1})} = H^t$ and 
        $G^t|_{F_{t}(G^{t-1})^c} = G^{t-1}.$
    \end{enumerate}
\end{enumerate}
Then, $\TV(\law(G^T), \RGGsphere) = o(1).$ Furthermore, with high probability, $F_t(G^{t-1})\subseteq F_{t+1}(G^{t})$ for each $t$ and there exists some $T_1 = \Theta (\frac{\log d}{\log \kappa})$ such that 
\begin{align*}
    |F_t(G^{t-1})|\le \begin{cases}\begin{aligned}
                &C_D n^2p\sqrt{\log 1/p}\times  \Big(\frac{np\log^6n}{d}\Big)^{{\frac{3^{t-1}}{2^t}}} \text{ for }t \le T_1\\[5pt]
        &C_Dn^{3/2}p^{1/2}\sqrt{\log 1/p}\times \Big(\frac{\log^2 n}{\sqrt{d}}\Big)^{t-T_1}\text{ for }t > T_1.
    \end{aligned}
    \end{cases}
\end{align*}
\end{theorem}
One should think of $F_t(G^{t-1})$ as the set of defect edges after $t$ rounds of the coupling algorithm. Note that we always resolve the defect edges by planting an \ERspace graph on them. Hence, the above observation produces a description of $\RGGsphere$ as a recursively planted \ER. Each next planted \ERspace is (doubly-)exponentially smaller.

When $\kappa = n^\epsilon$ for some constant $\epsilon>0,$ in the above statement $T = O (1/\epsilon).$ Hence, to construct a random geometric graph in dimension $n^{1+\epsilon}p,$ one only needs to recursively plant constantly many times an \ERspace graph in an \ERspace graph according to the functions $F_i.$ 



%% file: Coupling.tex
\section{Coupling Random Geometric Graphs and \texorpdfstring{\ER}{ER}}
\label{sec:coupling}

In this section, we make the intuition in \cref{sec:couplingnontecnhical} rigorous. 
We first define the map $\rho^b_a(z)$ in \cref{sec:fliporientation} and analyse some of its properties. Then, in \cref{sec:distributional} we prove \cref{lem:flippreservesiid} and in \cref{sec:martingaleanalysis} we prove \cref{lem:flippreservesinnerproducts}. In \cref{sec:algorithmictostochastic} we deduce \cref{thm:stochasticcoupling} from \cref{thm:algorithmiccoupling}.

\subsection{The Flip-Orientation Map}
\label{sec:fliporientation}

Define the unique monotone map $\phi:[-1,1]\longrightarrow[-1,1]$ such that 
$\phi(\tau^p_d) = \tau^p_d, \phi(1) = -1, \phi(-1) = 1$ and $\phi:[-1,\tau^p_d]\longrightarrow [\tau^p_d,1],$ 
$\phi:[\tau^p_d,1]\longrightarrow [-1,\tau^p_d]$ is measure-preserving with respect to $\sphericaloned|_{(-1,\tau^p_d]}$ and 
$\sphericaloned|_{[\tau^p_d,1)}.$ In particular, $\phi\circ\phi = \mathsf{identity}.$ 

For $b\in \{0,1\}$, define the map $\rho^b_a(z)$ which (possibly) flips $z\in \dsphere$ around $a\in \dsphere$ so that $b= \indicator[\langle a,z\rangle\ge\tau^p_d].$ The vector $z$ is transformed
in the $a$-direction  so that the projection on $a$ changes to $\phi(\langle a,z\rangle)$ instead of $\langle a,z\rangle$ if  $\indicator[\langle \rho_a^b(z),a\rangle \ge \tau^p_d] \neq b.$ Formally, 
\begin{equation}
    \begin{split}
        & \rho^b_a(z)\coloneqq
        \begin{cases}
        z\quad \text{ if }\indicator[\langle a,z\rangle \ge \tau^p_d] = b,\\
        (z - \langle a,z\rangle a)\sqrt{\frac{1 - \phi^2(\langle a,z\rangle)}{1 - \langle a,z\rangle^2}} + 
        \phi(\langle a,z\rangle) a \quad\text{ if }\indicator[\langle a,z\rangle \ge \tau^p_d] \neq b.
        \end{cases}
    \end{split}
\end{equation}
To simplify notation, we will denote 
\begin{align*}
     \sqrt{\frac{1 - \phi^2(\langle a,z\rangle)}{1 - \langle a,z\rangle^2}} &\coloneqq 1 + \psi(\langle a,z\rangle),\\
     s^b(\langle a,z\rangle) &\coloneqq\indicator\Big[\indicator[\langle a,z\rangle \ge \tau^p_d] \neq b\Big],\text{ and }\\
     \kappa(\langle a,z\rangle ) &\coloneqq  \big(- \langle a,z\rangle - \psi(\langle a,z\rangle)\langle a,z\rangle + \phi(\langle a,z\rangle)\big).
\end{align*}
    Thus, 
    \begin{equation}
    \label{eq:expandedrho}
    \begin{split}
     \rho^b_a(z) &= 
        z + \indicator\Big[\indicator[\langle a,z\rangle \ge \tau^p_d] \neq b\Big] \times \Big( \psi(\langle a,z\rangle)z  + \kappa(\langle a,z\rangle)a\Big)\\
    &= z + s^b(\langle a,z\rangle)\psi(\langle a, z\rangle) z + 
    s^b(\langle a,z\rangle)\kappa(\langle a, z\rangle) a.
    \end{split}
    \end{equation}

We make several simple but key observations.

\begin{observation}
\label{obs:unifromfixeddirection}  Let $a\in \dsphere$.
If $b\sim\Bernoulli(p), z\sim \unif(\dsphere)$, and $a,b$, and $z$ are independent, then $\rho^b_a(Z)\sim \unif(\dsphere).$
\end{observation}
\begin{proof} First, we will show that the projection $\langle \rho_a^b(z), a\rangle$ of  $\rho_a^b(z)$ on $a$ has distribution $\sphericaloned.$ By the independence of $a$ and $z,$ the distribution of the projection of  $z$ on $a$ is $\sphericaloned,$ which is the same as the mixture $(1-p)\sphericaloned|_{[-1,\tau^p_d]} + p\sphericaloned|_{[\tau^p_d,1]}.$ As $\phi$ measure-preservingly maps $\sphericaloned|_{[-1,\tau^p_d]}$ and $ \sphericaloned|_{[\tau^p_d,1]}$ to each other, $\langle a, \rho_a^1(z)\rangle\sim \sphericaloned|_{[\tau^p_d,1]}$ and $ \langle a, \rho_a^1(z)\rangle\sim \sphericaloned|_{[-1,\tau^p_d]}.$ As
 $b$ is independent of $a,z,$ $\langle a,\rho_a^b(z)\rangle$  has distribution $(1-p)\sphericaloned|_{[-1,\tau^p_d]} + p\sphericaloned|_{[\tau^p_d,1]} = \sphericaloned$. The distribution of $\rho_a^b(z)$
in the $a$ direction is correct.

The distribution in the orthogonal space to $a$ has the same direction as $z$ (which is uniform in this orthogonal space). The norm is uniquely determined by $\langle \rho_a^b(z),a\rangle$ as $\rho_a^b(z)$ is unit norm. By the previous paragraph, $\langle \rho_a^b(z), a\rangle$ has the correct distribution, hence so does the norm and so too the distribution in the orthogonal space.
\end{proof}

\begin{observation}
\label{obs:signexpectation}  
If $b\sim\Bernoulli(p), z\sim \unif(\dsphere)$ and $a,b,z$ are independent, then $\expect[s^b(\langle a, z\rangle)]=p.$
\end{observation}
\begin{proof} $b$ and $\indicator[\langle a, z\rangle\ge \tau^p_d]$ are i.i.d. $\Bernoulli(p).$    
\end{proof}

\begin{proposition}
\label{prop:elemntarybounds}
Suppose that 
$|\langle a, z \rangle|\le\frac{100C_A\sqrt{\log n}}{\sqrt{d}}$ where $C_A$ is the constant from \cref{lemma:sphericalcdf}.
Then, for some absolute constant $C$:
\begin{enumerate}
    \item $|\phi(\langle a, z \rangle)|\le\frac{C\sqrt{\log n}}{\sqrt{d}},$
    \item $|\psi(\langle a, z\rangle)|\le \frac{C\log n}{d},$ and
    \item  $|\kappa(\langle a, z\rangle)|\le\frac{C\sqrt{\log n}}{\sqrt{d}}.$
\end{enumerate}
\end{proposition}
\begin{proof}
For item 1, if $\langle a, z\rangle\ge \tau^p_d,$ then clearly $|\phi(\langle a, z \rangle)|\le|\langle a, z\rangle|.$ Else, 
$|\phi(\langle a, z \rangle)|\le\big|\phi(-\frac{100C_A\sqrt{\log n}}{\sqrt{d}})\big|.$ 
Note that $$\prob_{x\sim \sphericaloned|_{[-1,\tau^p_d]}}\big[x\le -\tfrac{100C_A\sqrt{\log n}}{\sqrt{d}}\big]  = 
\prob_{x\sim \sphericaloned}\big[x\le -\tfrac{100C_A\sqrt{\log n}}{\sqrt{d}}\big]/(1-p) \ge n^{-K},$$ where $K$ is a constant independent of $n,d.$ Hence, $\prob_{x\sim \sphericaloned|_{[\tau^p_d,1]}}\big[x\ge \phi(-\frac{100C_A\sqrt{\log n}}{\sqrt{d}})\big]\ge n^{-K},$ so 
$\prob_{x\sim \sphericaloned}\big[x\ge \phi(-\frac{100C_A\sqrt{\log n}}{\sqrt{d}})\big]\ge pn^{-K}\ge n^{-K-2}.$ This is enough by \cref{lemma:sphericalcdf}.
 
 For item 2, $$\sqrt{\frac{1 - \phi^2(\langle a,z\rangle)}{1 - \langle a,z\rangle^2}}\in \Big[\sqrt{1 - 2(\phi^2(\langle a,z\rangle) + \langle a,z\rangle^2)}, \sqrt{1 + 2(\phi^2(\langle a,z\rangle) + \langle a,z\rangle^2)}\Big].$$ As $2(\phi^2(\langle a,z\rangle) + \langle a,z\rangle^2)\le C_1\log n/d$ for some constant $C_1,$ we have that $\Big|\sqrt{\frac{1 - \phi^2(\langle a,z\rangle)}{1 - \langle a,z\rangle^2}} - 1\Big| = O(\log n/d).$ 
 
 Item 3 Follows from the definition of $\kappa$ and items 1 and 2.
\end{proof}

\subsection{Proof of \texorpdfstring{\cref{lem:flippreservesiid}}{Distributional Correctness}}
\label{sec:distributional}
\begin{proof}[Proof of \cref{lem:flippreservesiid}] The proof will proceed by induction on $(s,i_s)$ in the lexicographic order. Suppose that we have proven the statement for some $(s, \ell).$ We consider two cases.

\textbf{Case 1:} The next lexicographic term is $(s+1, 0),$ i.e. $\ell = s-1.$ Hence,
$\Vfin{1}, \Vfin{2},\ldots, \Vfin{s}\iidsim\unif(\dsphere)$ by induction. By definition, $V^{(0)}_{s+1}$ is independent of these vectors and also uniform on $\dsphere.$

\textbf{Case 2:} Now, suppose that $\ell < s-1.$ Hence, $\Vfin{1}, \Vfin{2},\ldots, \Vfin{s-1},V^{(\ell)}_s\iidsim\unif(\dsphere).$ We need to prove the statement for  $\Vfin{1}, \Vfin{2},\ldots, \Vfin{s-1},V^{(\ell+1)}_s= \rho^{H_{s,\ell+1}}_{\Vfin{\ell + 1}}(V^{(\ell)}_s).$ This follows from \cref{obs:unifromfixeddirection} since $H_{s, \ell + 1}, V^{(\ell)}_s$ are independent of $\{\Vfin{1}, \Vfin{2},\ldots, \Vfin{s-1}\}$ (and $\ell \le s-2$). 
\end{proof}

Using \cref{lemma:sphericalcdf}, we immediately conclude the following statement.

\begin{corollary}
\label{cor:boundedinnerproducts}
With high probability $1 - \frac{1}{n^{100}}$, for all $1\le i < j\le n,$
$0\le u \le i-1, 0 \le v\le j-1,$
$$
\big|\langle V^{(u)}_i, V^{(v)}_j\rangle\big|\le \frac{100C_A\sqrt{\log n}}{\sqrt{d}},
$$
and all of the remaining bounds in \cref{prop:elemntarybounds} also hold.
\end{corollary}
\subsection{Proof of \texorpdfstring{\cref{lem:flippreservesinnerproducts}}{Lemma 3.2.}}
\label{sec:martingaleanalysis}
\paragraph{Step 1: Expanding the telescoping sum.}
We begin by expanding the terms of \cref{eq:telescopingsum} using the expression in \cref{eq:expandedrho}. The goal is to isolate the ``martingale violations'' discussed in \ref{itm:v1} and \ref{itm:v2}.

\begin{align}
&
\sum_{u = i+1}^{j-1}
\langle \Vfin{i}, 
V^{(u)}_j- 
V^{(u-1)}_j
\rangle \notag\\
& =
\sum_{u = i+1}^{j-1}
s^{H_{ju}}
(\langle {\Vfin{u}},V^{(u-1)}_j\rangle)
\Big(
\psi(\langle 
\Vfin{u}, 
V^{(u-1)}_j
\rangle)
\langle 
\Vfin{i}, 
V^{(u-1)}_j
\rangle+ 
\kappa(\langle 
\Vfin{u}, 
V^{(u-1)}_j
\rangle)
\langle
\Vfin{u},
\Vfin{i}
\rangle
\Big) \notag\\
& = \sum_{u = i+1}^{j-1}
s^{H_{ju}}(
\langle {\Vfin{u}},V^{(u-1)}_j\rangle)
\psi(\langle 
\Vfin{u}, 
V^{(u-1)}_j
\rangle)
\langle 
\Vfin{i}, 
V^{(u-1)}_j
\rangle \label{eq:violations1}\tag{Violations 1}\\
& \quad+ \sum_{u = i+1}^{j-1}
s^{H_{ju}}(
\langle {\Vfin{u}},V^{(u-1)}_j\rangle)
\kappa(\langle 
\Vfin{u}, 
V^{(u-1)}_j
\rangle)
\langle
\Vfin{u},
V^{(u-1)}_j\rangle
\langle
V^{(u-1)}_j,
\Vfin{i}
\rangle \label{eq:violations2}\tag{Violations 2}\\
& \quad + \sum_{u = i+1}^{j-1}
s^{H_{ju}}(
\langle {\Vfin{u}},V^{(u-1)}_j\rangle) 
\kappa(\langle 
\Vfin{u}, 
V^{(u-1)}_j
\rangle)
\langle
\Vfin{u},
(\Vfin{i})^{\perp V^{(u-1)}_j}
\rangle. \label{eq:martingalepart}\tag{Martingale Part}
\end{align}
Above, $(\Vfin{i})^{\perp V^{(u-1)}_j}$ denotes the part of $\Vfin{i}$ perpendicular to $ V^{(u-1)}_j.$

We have decomposed the sum into three parts. The first one corresponds to martingale violations in \ref{itm:v1}, the second to the violations in \ref{itm:v2}. These will be bounded in Steps~2 through~4 below. The last part is a true martingale and will be analyzed in Steps~5 and~6.

To bound the martingale violations, we first show that with high probability there are few  nonzero terms and then show that each of these terms is small.

\paragraph{Step 2: The number of non-zero terms.}
\begin{lemma}
\label{lem:numberof}
With high probability $1- {1}/{n^{100}},$ the following inequality holds for all $i,j$ simultaneously:
    $$
\sum_{u = i+1}^{j-1}
s^{H_{ju}}(
\langle {\Vfin{u}},V^{(u-1)}_j\rangle)\le 200C_A\max (np, \log n).$$ 
\end{lemma}
\begin{proof} We first show that
$\displaystyle \Big\{
M_u =
s^{H_{ju}}(
\langle {\Vfin{u}},V^{(u-1)}_j\rangle) - p
\Big\}_{u= i+1}^{j-1}
$
is a martingale difference sequence with respect to the filtration 
\begin{equation}
\label{eq:filtration}
\begin{split}
\mathcal{F}_u \coloneqq   
\sigma\Big\{ 
\Vfin{i}, \Vfin{i+1}, \ldots, \Vfin{u},\quad
 H_{j,i+1}, H_{j,i+2}, \ldots, H_{j,u},\quad
 V^{(i)}_j
\Big\}.
\end{split}
\end{equation}
Measurability with respect to this filtration holds because $V^{(u-1)}_j$ is a deterministic function of\linebreak $V^{(i)}_j, \Vfin{i},\ldots, \Vfin{u-1}$ and $H_{j,i+1}, H_{j,i+2}, \ldots, H_{j,u-1}$. If $H_{ju}, \Vfin{u}$ are independent of $\mathcal{F}_{u-1},$ the zero expectation follows immediately from  \cref{obs:signexpectation} as 
$
    \expect\Big[
    s^{H_{ju}}
    (\langle {\Vfin{u}},V^{(u-1)}_j\rangle) - p
    \Big|\mathcal{F}_{u-1}\Big]= 0.$
The independence of  $H_{ju}, \Vfin{u}$ holds because the sigma algebra $\mathcal{F}_{u-1}$ is coarser than the sigma algebra 
$$
\sigma\{ 
\Vfin{1}, \Vfin{2}, \ldots, \Vfin{u-1},\quad\quad
 H_{j,0},H_{j,1}, \ldots, H_{j,u-1},\quad\quad
 V^{(0)}_j
\}.
$$
However, $\Vfin{1}, \Vfin{2}, \ldots, \Vfin{u-1},\Vfin{u},V^{(0)}_j$ are independent by \cref{lem:flippreservesiid}. And also $H_{j,0}, \ldots, H_{j,u}$ are independent from  $\Vfin{i},\Vfin{i+1}, \ldots, \Vfin{u-1},\Vfin{u},V^{(0)}_j.$ Hence, the collection 
$$\{
\Vfin{1}, \Vfin{2}, \ldots, \Vfin{u-1},\Vfin{u},V^{(0)}_j, 
H_{j,0},H_{j,1}, \ldots, H_{j,u-1},H_{j,u}
\}$$
is jointly independent and so $H_{ju}, \Vfin{u}$ are independent of the sigma algebra generated by\linebreak $\{ 
\Vfin{i},\Vfin{i+1}, \ldots, \Vfin{u-1},
 H_{j,0},H_{j,1}, \ldots, H_{j,u-1},
 V^{(0)}_j
\},$ which makes $V^{(i-1)}_j$ measurable.

Going back to the problem, we have the martingale difference sum 
 $$
Q= \sum_{u = i+1}^{j-1}
s^{H_{ju}}(
\langle {\Vfin{u}},V^{(u-1)}_j\rangle) - p.
$$
Each term is bounded by 1 and has conditional variance $p-p^2 \le p.$ By \cref{thm:mgbernstein} for $L = np$ and $a = 1,$ the inequality
$
\prob\big[|Q|\ge z\big]\le
2\exp(- \frac{z^2/2}{np + z/3})
$ holds.
Take $z= 100\max (\sqrt{np\log n}, \log n).$ The number of non-zero terms is $Q + np \le np + 100C_A\max (\sqrt{np\log n}, \log n)\le 200C_A\max(np,\log n).$
\end{proof}

\paragraph{Step 3: Martingale violations \ref{itm:v1}.}
We now deal with the term \eqref{eq:violations1}. Recall from \ref{itm:v1} that this term appears due to the rescaling in the direction orthogonal to $\Vfin{u}$ when flipping $V^{(u-1)}_j$ with respect to $\Vfin{u}.$
For each $u,$ by \cref{prop:elemntarybounds,cor:boundedinnerproducts}, with high probability $1 - 1/n^{100}$
$$\big|\psi(\langle 
\Vfin{u}, 
V^{(u-1)}_j
\rangle)
\langle 
\Vfin{i}, 
V^{(u-1)}_j
\rangle\big| \le C'\frac{\log^{3/2}n}{d^{3/2}}$$ for some absolute constant $C'.$ Combining with \cref{lem:numberof},
with probability $1 - O(1/n^{100}),$
\begin{align*}
& \Big|\sum_{u = i+1}^{j-1}
s^{H_{ju}}(
\langle {\Vfin{u}},V^{(u-1)}_j\rangle) 
\psi(\langle 
\Vfin{u}, 
V^{(u-1)}_j
\rangle)
\langle 
\Vfin{i}, 
V^{(u-1)}_j
\rangle\Big|\\
& \le 
\sum_{u = i+1}^{j-1}
s^{H_{ju}}(
\langle {\Vfin{u}},V^{(u-1)}_j\rangle) 
\Big|\psi(\langle 
\Vfin{u}, 
V^{(u-1)}_j
\rangle)
\langle 
\Vfin{i}, 
V^{(u-1)}_j
\rangle\Big|  \le 
C'\max(np,\log n)\frac{\log^{3/2}n}{d^{3/2}}
\end{align*}
When $d=\Omega(\max(np,\log n)),$ we have ${d^{-3/2}}{\max(np,\log n)\log^{3/2}n} = {O({d^{-1}}\max(\sqrt{np}, \sqrt{\log n})(\log n)^{3/2}}).$

\paragraph{Step 4: Martingale violations \ref{itm:v2}.} Next, we consider the term in \eqref{eq:violations2}. As in \ref{itm:v2}, it appears due to fact that $V^{(u-1)}_j$ and $\Vfin{i}$ are not independent.
By \cref{prop:elemntarybounds,cor:boundedinnerproducts}, with high probability $1 - 1/n^{100}$,
\begin{equation*}
    \Big|\kappa(\langle 
\Vfin{u}, 
V^{(u-1)}_j
\rangle)
\langle
\Vfin{u},
V^{(u-1)}_j\rangle
\langle
V^{(u-1)}_j,
\Vfin{i}
\rangle\Big| \le C'\frac{\log^{3/2} n}{d^{3/2}}.
\end{equation*}
Using \cref{lem:numberof} we conclude that with probability $1 - O(1/n^{100})$,
\begin{align*}
    & \Big|
    \sum_{u = i+1}^{j-1}
s^{H_{ju}}(
\langle {\Vfin{u}},V^{(u-1)}_j\rangle)
\kappa\big(\langle 
\Vfin{u}, 
V^{(u-1)}_j
\rangle\big)
\langle
\Vfin{u},
V^{(u-1)}_j\rangle
\langle
V^{(u-1)}_j,
\Vfin{i}
\rangle
    \Big| \le 
    \frac{C'\max(np,\log n)\log^{3/2}n}{d^{3/2}}.
\end{align*}
Again, when $d= \Omega(\max(np,\log n)),$ this expression is of order $O(\max(\sqrt{np}, \sqrt{\log n})\frac{(\log n)^{3/2}}{d}).$
\paragraph{Step 5: Truncating the martingale differences.} The remaining sum 
$$
\sum_{u = i+1}^{j-1}
s^{H_{ju}}(
\langle {\Vfin{u}},V^{(u-1)}_j\rangle) 
\kappa(\langle 
\Vfin{u}, 
V^{(u-1)}_j
\rangle)
\langle
\Vfin{u},
(\Vfin{i})^{\perp V^{(u-1)}_j}
\rangle
$$
is a sum of martingale differences. However, we need to truncate it so that we have a small uniform bound to apply Freedman's inequality \cref{thm:mgbernstein}.
To this end, we write
\begin{align*}
& \sum_{u = i+1}^{j-1}
s^{H_{ju}}(
\langle {\Vfin{u}},V^{(u-1)}_j\rangle) 
\kappa(\langle 
\Vfin{u}, 
V^{(u-1)}_j
\rangle)
\langle
\Vfin{u},
(\Vfin{i})^{\perp V^{(u-1)}_j}
\rangle\\
& = 
\sum_{u = i+1}^{j-1}
s^{H_{ju}}(
\langle {\Vfin{u}},V^{(u-1)}_j\rangle) 
\kappa(\langle 
\Vfin{u}, 
V^{(u-1)}_j
\rangle)
\langle
\Vfin{u},
(\Vfin{i})^{\perp V^{(u-1)}_j}
\rangle\times\\
&\quad\quad\quad\quad\times
\indicator\Big[\big|\langle 
\Vfin{u}, 
V^{(u-1)}_j
\rangle\big|\le \tfrac{100C_A\sqrt{\log n}}{\sqrt{d}}\Big]
\indicator\Big[\big| 
\langle
\Vfin{u},
(\Vfin{i})^{\perp V^{(u-1)}_j}
\rangle\big|\le \tfrac{200C_A\sqrt{\log n}}{\sqrt{d}}\Big]\\
&\quad + 
\sum_{u = i+1}^{j-1}
s^{H_{ju}}(
\langle {\Vfin{u}},V^{(u-1)}_j\rangle)
(V^{(u-1)}_j) 
\kappa\big(\langle 
\Vfin{u}, 
V^{(u-1)}_j
\rangle\big)
\langle
\Vfin{u},
(\Vfin{i})^{\perp V^{(u-1)}_j}
\rangle\\
&\quad\quad\quad\quad\quad\times
\indicator\Big[\big|\langle 
\Vfin{u}, 
V^{(u-1)}_j
\rangle\big|> \tfrac{100C_A\sqrt{\log n}}{\sqrt{d}}\text{ or }\big| 
\langle
\Vfin{u},
(\Vfin{i})^{\perp V^{(u-1)}_j}
\rangle\big|> \tfrac{200C_A\sqrt{\log n}}{\sqrt{d}}\Big]
\end{align*}
Observe that 
\begin{align*}
    & \big|\langle
\Vfin{u},
(\Vfin{i})^{\perp V^{(u-1)}_j}
\rangle\big| = 
\big|\langle
\Vfin{u},
\Vfin{i}
\rangle-\langle
\Vfin{u},
V^{(u-1)}_j\rangle
\langle
V^{(u-1)}_j,
\Vfin{i}
\rangle\big|\\
&
\le 
\big|\langle
\Vfin{u},
\Vfin{i}
\rangle\big| + 
\big|\langle
\Vfin{u},
V^{(u-1)}_j\rangle
\langle
V^{(u-1)}_j,
\Vfin{i}
\rangle\big| \le \big|\langle
\Vfin{u},
\Vfin{i}
\rangle\big| + 
\big|
\langle
\Vfin{u},
V^{(u-1)}_j\rangle\big|.
\end{align*} 
Hence, by \cref{cor:boundedinnerproducts},
$$\prob\Big[\big|\langle 
\Vfin{u}, 
V^{(u-1)}_j
\rangle\big|> \tfrac{100C_A\sqrt{\log n}}{\sqrt{d}}\text{ or }\big|\langle 
\Vfin{u},
(\Vfin{i})^{\perp V^{(u-1)}_j}
\rangle\big|> \tfrac{200C_A\sqrt{\log n}}{\sqrt{d}}\Big]\le 1 - \frac{1}{n^{100}}.$$ 
Thus, with probability $1 -O(1/n^{100}),$ the second sum is 0 for all $i,j.$ 

\paragraph{Step 6: Martingale analysis.}
All that is left to bound is the sum
\begin{equation}
\label{eq:FINALMGDIFFERENCES}
\begin{split}
&\sum_{u = i+1}^{j-1}
s^{H_{ju}}(
\langle {\Vfin{u}},V^{(u-1)}_j\rangle)
\kappa\big(\langle 
\Vfin{u}, 
V^{(u-1)}_j
\rangle\big)
\langle
\Vfin{u},
(\Vfin{i})^{\perp V^{(u-1)}_j}
\rangle\times\\
&\quad\quad\quad\quad\times
\indicator\Big[|\langle 
\Vfin{u}, 
V^{(u-1)}_j
\rangle|\le \tfrac{100C_A\sqrt{\log n}}{\sqrt{d}}\Big]
\indicator\Big[\big|\langle 
\langle
\Vfin{u},
(\Vfin{i})^{\perp V^{(u-1)}_j}
\rangle\big|\le \tfrac{100C_A\sqrt{\log n}}{\sqrt{d}}\Big].
\end{split}
\end{equation}
Let $K_u$ be the term corresponding to $u$ above. We claim that $\{K_u\}_{u = i+1}^{j-1}$ is a martingale difference sequence with respect to the filtration in \cref{eq:filtration}. Measurability follows in the same as way as for $\{M_u\}_{u = i+1}^{j-1}$ defined in \cref{lem:numberof}. We also need to show the martingale property.

\paragraph{Step 6.1. Martingale property.} For brevity of notation, denote 
\begin{align*}
&\indicator(u,j) \coloneqq \indicator\Big[|\langle 
\Vfin{u}, 
V^{(u-1)}_j
\rangle|\le \tfrac{100C_A\sqrt{\log n}}{\sqrt{d}}\Big]\\
&\indicator(u,i,\perp j) \coloneqq \indicator\Big[| 
\langle
\Vfin{u},
(\Vfin{i})^{\perp V^{(u-1)}_j}
\rangle|\le \tfrac{100C_A\sqrt{\log n}}{\sqrt{d}}\Big].
\end{align*} 
We need to show that 
$$
\expect\Big[
s^{H_{ju}}(
\langle {\Vfin{u}},V^{(u-1)}_j\rangle) 
\kappa(\langle 
\Vfin{u}, 
V^{(u-1)}_j
\rangle)
\langle
\Vfin{u},
(\Vfin{i})^{\perp V^{(u-1)}_j}\rangle
\indicator(u,j)
\indicator(u,i,\perp j)
\Big|\mathcal{F}_{u-1}
\Big]= 0.
$$
This follows immediately from the fact that $\Vfin{u}$ is independent of 
$\cF_{u-1},H_{ju}.$ Hence, 
\begin{align*}
& \expect\Big[
s^{H_{ju}}(
\langle {\Vfin{u}},V^{(u-1)}_j\rangle)
\kappa(\langle 
\Vfin{u}, 
V^{(u-1)}_j
\rangle)
\langle
\Vfin{u},
(\Vfin{i})^{\perp V^{(u-1)}_j}\rangle \indicator(u,j)
\indicator(u,i,\perp j)
\Big|\mathcal{F}_{u-1}
\Big]\\
& = 
\expect\Big[
\expect\Big[
s^{H_{ju}}(
\langle {\Vfin{u}},V^{(u-1)}_j\rangle) 
\kappa(\langle 
\Vfin{u}, 
V^{(u-1)}_j
\rangle)
\langle
\Vfin{u},
(\Vfin{i})^{\perp V^{(u-1)}_j}\rangle\times\\
&\quad\quad\quad\quad\times\indicator(u,j)
\indicator(u,i,\perp j)
\Big|\mathcal{F}_{u-1}, H_{ju}, 
\langle \Vfin{u}, V^{(u-1)}_j\rangle
\Big]\; \Big| \mathcal{F}_{u-1}\;
\Big]\\
& =
\expect\Big[
s^{H_{ju}}(
\langle {\Vfin{u}},V^{(u-1)}_j\rangle)
\kappa(\langle 
\Vfin{u}, 
V^{(u-1)}_j
\rangle)
\indicator(u, j)\times\\
&\quad\quad\quad\quad\times
\expect\Big[
\langle
\Vfin{u},
(\Vfin{i})^{\perp V^{(u-1)}_j}\rangle
\indicator(u,i,\perp j)
\Big|\mathcal{F}_{u-1}, H_{ju}, 
\langle \Vfin{u}, V^{(u-1)}_j\rangle
\Big]\; \Big| \mathcal{F}_{u-1}\;
\Big]
\end{align*}
We used the fact that $s^{H_{ju}}(
\langle {\Vfin{u}},V^{(u-1)}_j\rangle), \indicator(u, j) $ are functions of only $H_{ju}, \langle \Vfin{u}, V^{(u-1)}_j\rangle.$ Now, the inner expectation is equal to 0 as the map $$V +\langle \Vfin{u}, V^{(u-1)}_j\rangle V^{(u-1)}_{j}\longrightarrow -V +\langle \Vfin{u}, V^{(u-1)}_j\rangle V^{(u-1)}_{j}\,,$$ where $V$ is the component of
$\Vfin{u}$ orthogonal to $V^{(u-1)}_j$, is measure-preserving under the sigma algebra $\sigma\{\mathcal{F}_{u-1}, H_{ju}, 
\langle \Vfin{u}, V^{(u-1)}_j\rangle\}$ and acts by
\begin{align*}
& \langle
\Vfin{u},
(\Vfin{i})^{\perp V^{(u-1)}_j}\rangle
\indicator\Big[|\langle 
\langle
\Vfin{u},
(\Vfin{i})^{\perp V^{(u-1)}_j}
\rangle|\le \tfrac{100C_A\sqrt{\log n}}{\sqrt{d}}\Big]\\
&\quad\quad\quad\quad\longrightarrow
- \langle
\Vfin{u},
(\Vfin{i})^{\perp V^{(u-1)}_j}\rangle
\indicator\Big[|\langle 
\langle
\Vfin{u},
(\Vfin{i})^{\perp V^{(u-1)}_j}
\rangle|\le \tfrac{100C_A\sqrt{\log n}}{\sqrt{d}}\Big].
\end{align*}

\paragraph{Step 6.2. Martingale concentration.} We end by applying \cref{thm:mgbernstein}. Recall that
$$
K_u =
s^{H_{ju}}(
\langle {\Vfin{u}},V^{(u-1)}_j\rangle)
\kappa\big(\langle 
\Vfin{u}, 
V^{(u-1)}_j
\rangle\big)
\langle
\Vfin{u},
(\Vfin{i})^{\perp V^{(u-1)}_j}\rangle
\indicator(u,j)
\indicator(u,i,\perp j)
$$
is bounded deterministically by $C''\frac{\log n}{d}$ for some absolute constant $C''$ due to the indicators and the bound on $\kappa$ from \cref{prop:elemntarybounds}. The conditional second moment is at most $(C'')^2p\times \frac{\log^2 n}{d^2}:$ 
\begin{align*}
& \expect[K_u^2|\cF_{u-1}]\\
& = 
\expect\Big[
\Big(s^{H_{ju}}(
\langle {\Vfin{u}},V^{(u-1)}_j\rangle) \Big)^2
\kappa(\langle 
\Vfin{u}, 
V^{(u-1)}_j
\rangle)^2
\langle
\Vfin{u},
(\Vfin{i})^{\perp V^{(u-1)}_j}\rangle^2
\indicator(u,j)
\indicator(u,i,\perp j)\Big|\cF_{u-1}
\Big]\\
& \le 
\expect\Big[
\Big(s^{H_{ju}}(
\langle {\Vfin{u}},V^{(u-1)}_j\rangle) \Big)^2
(C'')^2\frac{\log^2 n}{d^2}
\Big|\cF_{u-1}
\Big] \le (C'')^2p\times \frac{\log^2 n}{d^2}.
\end{align*}
Altogether, 
by \cref{thm:mgbernstein},
\begin{align*}
 \prob\Big[\sum_{u = i+1}^{j-1}K_u\ge z\Big]&\le 
\exp\Big( - \frac{z^2/2}{(C'')^2(j-i-1)pd^{-2}\log^2n + zC''d^{-1}\log n/3}\Big)\\
& = 
\exp\Big( - \frac{z^2}{O(np\log^2n/d^2) + O(z\log n/d))} \Big).
\end{align*}
Taking $z = \frac{c\max(\sqrt{np}, \sqrt{\log n})(\log n)^{3/2}}{d}$ 
for a large enough $c$ makes this last expression $O(1/n^{100}).$

\subsection{Deriving \texorpdfstring{\cref{thm:stochasticcoupling}}{Stochastic Dominance} from \texorpdfstring{\cref{thm:algorithmiccoupling}}{Algorithmic Coupling}}
\label{sec:algorithmictostochastic}
\begin{proof}[Proof of \cref{thm:stochasticcoupling}] 
Suppose that one is given the coupling in \cref{thm:algorithmiccoupling} of $H\sim \ergraph$ with vectors $\Vfin{1}, \Vfin{2}, \ldots, \Vfin{n}\iidsim\unif(\dsphere)$ such that the second condition holds. Thus,
with high probability, all edges of the random geometric graph $G_+$ defined by 
$$
(G_+)_{ij} = 
\indicator\Big[\langle \Vfin{i}, \Vfin{j}\rangle \ge \tau^p_d - c\max(\sqrt{np}, \sqrt{\log n})\frac{(\log n)^{3/2}}{d}\Big]
$$
are also edges in $H$ defined by $H_{ji}.$ Note that $G_+$ is a sample from $\RGG(n,\dsphere, p_+)$ where
$$p_+\coloneqq \prob_{U,V\iidsim\unif(\dsphere)}[\langle U, V\rangle\ge \tau^p_d - c\max(\sqrt{np}, \sqrt{\log n})\frac{(\log n)^{3/2}}{d}].$$ From \cref{lemma:sphericalcdf}, $p_+ = p\Big(1 + O\Big(\max(\sqrt{np}, \sqrt{\log n})\frac{(\log n)^{2}}{\sqrt{d}}\Big)\Big).
$ Similarly, we consider 
$$
(G_-)_{ij} = 
\indicator\Big[\langle \Vfin{i}, \Vfin{j}\rangle \ge \tau^p_d + \max(\sqrt{np}, \sqrt{\log n})\frac{(\log n)^{3/2}}{d}\Big].
$$
Hence, we have a coupling of $H\sim \ergraph,G_+\sim  \RGG\Big(n,d,p\Big(1 + O\Big(\max(\sqrt{np}, \sqrt{\log n})\frac{(\log n)^{2}}{\sqrt{d}}\Big)\Big)\Big)$ and 
$G_-\sim  \RGG\Big(n,d,p\Big(1 - O\Big(\max(\sqrt{np}, \sqrt{\log n})\frac{(\log n)^{2}}{\sqrt{d}}\Big)\Big)\Big)$ such that $G_-\subseteq H\subseteq G_+$ with high probability. This is clearly enough. 
\end{proof}

%% file: RobustTesting.tex
\section{Robust Testing}
\label{sec:robusttesting}
\subsection{The SoS Program and Refutation}
\label{sec:testingsos}
We now give the details for the SoS program. Consider the following axiom set over $n$ $d$-dimensional variables $(Y_i)_{i = 1}^n$ stacked together in a $n\times d$ matrix $Y$ and indicator variables $w= (w_{ij})_{1\le i < j \le n}$:
\begin{align}
\mathcal{A}(Y,w) \coloneqq 
\begin{cases}
\langle Y_i, Y_i\rangle = 1\; \forall i,\\
w_{i,j}^2 = w_{i,j}\; \forall i\neq j,\\
(1-w_{i,j})\langle Y_i, Y_j\rangle^2 \le  (\tau_d^p)^2 \; \forall i\neq j,\\
\sum_{i,j}w_{i,j}\langle Y_i,Y_j\rangle^2\le \frac{2C_2 n^2 p\log(1/p)}{d}, 
\end{cases}
\end{align}
where $C_2$ is the constant from \cref{lemma:sphericalcdf}.
One should interpret the condition imposed by variables $\omega_{ij}$ as a tail bound on the empirical distribution of 
$\{\langle Y_i, Y_j\rangle^2\}_{1\le i <j \le n}.$ The constraint essentially reads 
$
\sum_{i,j\; :\; \langle Y_i, Y_j\rangle^2\ge (\tau^p_d)^2}\langle Y_i, Y_j\rangle^2\le 
\frac{2C_2n^2p\log(1/p)}{\sqrt{d}}
,
$ which holds w.h.p. when one takes $Y_i = V_i.$
\paragraph{Spectral Refutation Algorithm.} On input a $p$-centered adjacency matrix $A = A_G,$ try to prove in degree $12$ SoS that
$\Big(\sum_{i,j}A_{i,j}\langle Y_i,Y_j\rangle\Big)^2\le \frac{n^4p^2\log(1/p)}{8C_B^2d},$ where $C_B$ is the constant from \cref{lemma:sphericalcdf}.  If the algorithm fails, report that $G$ is a corrupted $\RGG(n,d,p).$ If it succeeds, report that it is a corruption of a graph with small operator norm of its $p$-centered adjacency matrix (e.g. $\ergraph$ for concreteness).

\subsection{Analysis of Corrupted \ER}
Let $A$ be the centered adjacency matrix of a corrupted $\ergraph.$ Let $A = U + S,$ where $U$ is the uncorrupted adjacency matrix and $S$ is defined by $S = A - U.$ The entries of $S$ are in $\{-1,0,1\}$ and at most $\epsilon p \binom{n}{2}$ of them are non-zero ($S$ is supported on corrupted entries). Hence, 
\begin{equation}
    \begin{split}
        &\Big(\sum_{i,j}A_{i,j}\langle Y_i,Y_j\rangle\Big)^2 = 
        \Big(
        \sum_{i,j}U_{i,j}\langle Y_i,Y_j\rangle + 
        \sum_{i,j}S_{i,j}\langle Y_i,Y_j\rangle
        \Big)^2\\
        & = 
        \Big(
        \sum_{i,j}U_{i,j}\langle Y_i,Y_j\rangle + 
        \sum_{i,j}S_{i,j}w_{i,j}\langle Y_i,Y_j\rangle + 
        \sum_{i,j}S_{i,j}(1-w_{i,j})\langle Y_i,Y_j\rangle
        \Big)^2\\
        & \le 
        4 \Big(
        \sum_{i,j}U_{i,j}\langle Y_i,Y_j\rangle\Big)^2
         +4 \Big(
        \sum_{i,j}S_{i,j}w_{i,j}\langle Y_i,Y_j\rangle\Big)^2
         +4
        \Big(
        \sum_{i,j}S_{i,j}(1-w_{i,j})\langle Y_i,Y_j\rangle\Big)^2,
    \end{split}
\end{equation}
where so far all inequalities are in SoS.
We bound each of the three terms separately.
\paragraph{Case 1) Uncorrupted value:} Note that $\sum_{i,j}U_{i,j}\langle Y_i,Y_j\rangle = \langle U, Y^TY\rangle.$ Since $U$ is the signed adjacency matrix of $\ergraph,$ with high probability $\lambda_{\max}(U) \le C\max(\sqrt{np},\log n)$ for some absolute constant $C$ by \cref{thm:secondeigenvalueofer}. Hence, 
$C\max(\sqrt{np},\log n)I - U\succeq 0$ and, thus, in SoS, 
\begin{align*}
    \langle U, YY^T\rangle \le 
    \langle C\max(\sqrt{np},\log n)I, YY^T\rangle =
    C\max(\sqrt{np},\log n)\times \trace(YY^T) = 
    C\max(\sqrt{np},\log n)\times n.
\end{align*}
Similarly, $ \langle U, YY^T\rangle \ge -C\max(\sqrt{np},\log n)\times n.$ Thus, in SoS, 
$\Big(\sum_{i,j}U_{i,j}\langle Y_i,Y_j\rangle\Big)^2  \le C^2\max({n^3p},n^2\log^2 n).$ 

\paragraph{Case 2)} Corrupted value with large entries. By Cauchy-Schwartz inequality, 
\begin{align*}
    &\Big(
        \sum_{i,j}S_{i,j}w_{i,j}\langle Y_i,Y_j\rangle\Big)^2\\
    &\le 
    \Big(
        \sum_{i,j}S_{i,j}^2\Big)
    \Big(\sum_{i,j}w_{i,j}^2\langle Y_i,Y_j\rangle^2\Big)\\
    &
    \le \Big(
        \sum_{i,j\; :\; (ij)\text{ corrupted}}1\Big)
    \Big(\sum_{i,j}w_{i,j}\langle Y_i,Y_j\rangle^2\Big)
    \\
    &\le \epsilon pn^2\times \frac{4C_2^2 n^2 p\log(1/p)}{d}.
\end{align*}

\paragraph{Case 3)} Corrupted value with small entries. Again, by Cauchy-Schwartz,
\begin{align*}
    &\Big(
        \sum_{i,j}S_{i,j}(1-w_{i,j})\langle Y_i,Y_j\rangle\Big)^2\\
    & = \Big(
        \sum_{i,j}S^3_{i,j}(1-w_{i,j})\langle Y_i,Y_j\rangle\Big)^2\\
    &\le 
    \Big(
        \sum_{i,j}S_{i,j}^2\Big)
    \Big(S_{i,j}^4\sum_{i,j}(1-w_{i,j})^2\langle Y_i,Y_j\rangle^2\Big)\\
    &\le 
    \epsilon p n^2\times 
    \sum_{1\le i <j \le n}S_{i,j}^4
    (\tau^p_d)^2\\
    &\le 
    \epsilon p n^2\times \epsilon p n^2\times (\tau^{p}_d)^2
    \le \frac{C_A^2\epsilon^2 n^4p^2\log (1/p)}{d},
\end{align*}
using \cref{lemma:sphericalcdf}.
Altogether, for some absolute constant $C',$ when $A$ is an $\epsilon$-corruption of $\ergraph,$ with high probability there is a degree 12 proof in SoS that
\begin{align}
\label{eq:corruptedERSoSValue}
\Big(\sum_{i,j}A_{i,j}\langle Y_i,Y_j\rangle\Big)^2\le 
C'\Big(\max(n^3p, n^2\log^2 n)
+  \frac{\epsilon n^4 p^2\log(1/p)}{d}\Big).
\end{align}

\subsection{Analysis of Corrupted RGG}
We will show that with high probability over $\RGG(n,d,p),$ for no $\epsilon$-corruption $A$ can one prove \linebreak $\Big(\sum_{i,j}A_{i,j}\langle Y_i,Y_j\rangle\Big)^2\le \frac{n^4p^2\log(1/p)}{d}.$ We simply need to exhibit a pair $(Y,w)$ that satisfies the axioms $\mathcal{A}(Y,w)$ and $\Big(\sum_{i,j}A_{i,j}\langle Y_i,Y_j\rangle\Big)^2\ge \frac{n^4p^2\log(1/p)}{8C_B^2d}.$

Let $Y_i = V_i, 1\le i \le n,$ where this is the true embedding that produced $A.$ Let $w_{i,j} = \indicator[\langle V_i, V_j\rangle^2\ge (\tau^p_d)^2].$ 

\paragraph{Feasibility.} Clearly, the first three axioms are satisfied. For the last one,
\begin{align*}
    &\expect_{V_1, V_2,\ldots, V_n\sim_{iid}\unif(\dsphere)}
    \Big[\sum_{i\neq j}
    \indicator[\langle V_i, V_j\rangle^2\ge (\tau^p_d)^2]\times 
    \langle V_i, V_j\rangle^2
    \Big]\\
    &\le n^2
    \expect_{V_1, V_2\sim_{iid}\unif(\dsphere)}
    \Big[
    \indicator[\langle V_1, V_2\rangle^2\ge (\tau^p_d)^2]\times 
    \langle V_1, V_2\rangle^2
    \Big]\\
    & = n^2
    \expect_{V_1, V_2\sim_{iid}\unif(\dsphere)}\Big[\langle V_1, V_2\rangle^2|\langle V_1, V_2\rangle^2\ge (\tau^p_d)^2\Big] \times 
    \prob\Big[\langle V_1, V_2\rangle^2\ge (\tau^p_d)^2\Big]\\
    & \le n^2p\times\frac{C_2^2\log(1/p)}{d} 
\end{align*}
by \cref{lemma:sphericalcdf}. Furthermore, the value concentrates extremely well as the collection of random variables $\{\indicator[\langle V_i, V_j\rangle^2\ge (\tau^p_d)^2]\times 
\langle V_i, V_j\rangle^2\}_{1\le i <j\le n}$ is pairwise independent! (Note that $\langle V_1, V_2\rangle$ and $\langle V_2, V_3\rangle$ are independent.) Hence, 
\begin{align*}
    &\var_{V_1, V_2,\ldots, V_n\sim_{iid}\unif(\dsphere)}
    \Big[\sum_{i\neq j}
    \indicator[\langle V_i, V_j\rangle^2\ge (\tau^p_d)^2]\times 
    \langle V_i, V_j\rangle^2
    \Big]\\
    &\le n^2
    \var_{V_1, V_2\sim_{iid}\unif(\dsphere)}
    \Big[
    \indicator[\langle V_1, V_2\rangle^2\ge (\tau^p_d)^2]\times 
    \langle V_1, V_2\rangle^2
    \Big]\\
    & \le n^2
    \expect_{V_1, V_2\sim_{iid}\unif(\dsphere)}\Big[\langle V_1, V_2\rangle^4|\langle V_1, V_2\rangle^2\ge (\tau^p_d)^2\Big] \times 
    \prob\Big[\langle V_1, V_2\rangle^2\ge (\tau^p_d)^2\Big]\\
    & \le n^2p\times\frac{C_4\log(1/p)^2}{d^2}= 
    o\bigg(\Big(n^2p\times\frac{C_2^2\log(1/p)}{d}\Big)^2\bigg).
\end{align*}
Hence, 
    $\sum_{i\neq j}
    \indicator[\langle V_i, V_j\rangle^2\ge (\tau^p_d)^2]\times 
    \langle V_i, V_j\rangle^2
    \le \frac{2C^2_2n^2p\log(1/p)}{d}$ w.h.p. by Chebyshev's inequality. 

\paragraph{Objective Value.}
Finally, we need to bound $\sum_{i,j}A_{i,j}\langle V_i,V_j\rangle.$ As in the \ERspace case, 
\begin{align*}
& \Big|\sum_{i,j}S_{i,j}\langle V_i,V_j\rangle\Big|\le 
2\sqrt{\Big(
        \sum_{i,j}S_{i,j}w_{i,j}\langle V_i,V_j\rangle\Big)^2 +
    \Big(
        \sum_{i,j}S_{i,j}(1-w_{i,j})\langle Y_i,Y_j\rangle\Big)^2
}\\
& \le 2\sqrt{\epsilon pn^2\times \frac{4C_2^2 n^2 p\log(1/p)}{d} + \frac{C_A^2\epsilon^2 n^4p^2\log (1/p)}{d}} \le  
\frac{C''\sqrt{\epsilon}n^2p\sqrt{\log(1/p)}}{\sqrt{d}}.
\end{align*}
Hence, we only need to show that $\sum_{i,j}U_{i,j}\langle V_i,V_j\rangle$ is large. Again, as the different summands are pairwise independent and this value concentrates well, we only consider the expectation:
\begin{align*}
&\expect\Big[\sum_{i,j}U_{i,j}\langle V_i,V_j\rangle\Big]\\
& = n(n-1)\expect\big[\indicator[\langle V_i, V_j\rangle\ge \tau^p_d]\langle V_i,V_j\rangle\big]\\
& = 
n(n-1)\expect\big[\langle V_i,V_j\rangle|\langle V_i,V_j\rangle\ge \tau^p_d\big]\times 
\prob\big[\langle V_i,V_j\rangle\ge \tau^p_d\big]\\
&\ge 
n(n-1)p
\expect\big[\langle V_i,V_j\rangle|\langle V_i,V_j\rangle\ge \tau^p_d\big]\ge \frac{n^2p\sqrt{\log(1/p)}}{2C_B\sqrt{d}},
\end{align*}
the last inequality by \cref{lemma:sphericalcdf}.
Hence, $\sum_{i,j}A_{i,j}\langle V_i,V_j\rangle\ge \frac{n^2p\sqrt{\log(1/p)}}{2C_B\sqrt{d}} - \frac{C''\sqrt{\epsilon}n^2p\sqrt{\log(1/p)}}{\sqrt{d}}$ and so 
$\Big(\sum_{i,j}A_{i,j}\langle V_i,V_j\rangle\Big)^2\ge 
\frac{n^4p^2\log(1/p)}{8C^2_Bd}
$ for all  $\epsilon\le K_r'$ for some absolute constant $K_r'.$
\subsection{Putting It All Together}
By the previous two sections, as long as $\frac{n^4p^2\log(1/p)}{8C^2_Bd}\ge C'\big(\max(n^3p, n^2\log^2 n)
+  \frac{\epsilon n^4 p^2\log(1/p)}{d}\big),$ the test succeeds with high probability. This occurs if $\epsilon\le K_r$ and $d\le K_r\min(np\log(1/p), n^2p^2\log(1/p)\log^{-2}n)$ for some absolute constant $C.$ 

\begin{remark}
\label{rmk:testingrobustly}
    The only property needed about the $\ergraph$ distribution in \cref{thm:robustestingsos} is that 
    for the centered adjacency matrix $U$ of a sample from $\ergraph,$ $|\lambda_{\max}(U)|\le c_1\frac{n^3p^2\log(1/p)}{d}.$ Hence, our result holds for robust testing between 
    $\RGGsphere$ and \emph{any} distribution $\mathcal{G}$ for which the $p$-centered adjacency matrix of a sample from $\mathcal{G}$ has operator norm at most $c_1\frac{n^3p^2\log(1/p)}{d}$ for some absolute constant $c_1.$
\end{remark}


%% file: Enumeration.tex
\section{Enumerating Geometric Graphs and the Support Size of {RGG}}
\label{sec:enumeration}
\begin{proof}[Proof of \cref{thm:enumeration}] Suppose that we are given some $d\ge \log^7 n$ (otherwise, the theorem statement follows immediately from the 2-dimensional case by \cite{MCDIARMID2011627}). 

Take $p = d/(n\log^{6.5}n) = \Omega(1/n).$ Now, \cref{thm:algorithmiccoupling} implies that there exists a coupling of $H \sim \ergraph$ and $V_1, V_2, \ldots, V_n\iidsim\unif(\dsphere)$ such that with high probability:
\begin{enumerate}
    \item $\indicator\Big[\langle V_i,  V_j\rangle\ge \tau^p_d - c\frac{\max(\sqrt{np},\log n)\log n^{3/2}}{d}\Big]\ge H_{i,j}$ for all $i,j.$
    \item $\indicator\Big[\langle V_i,  V_j\rangle\ge \tau^p_d + c\frac{\max(\sqrt{np},\log n)\log n^{3/2}}{d}\Big]\le H_{i,j}$ for all $i,j.$ 
\end{enumerate}

Thus, the random geometric graph $G$ defined by $G_{i,j} = \indicator\Big[\langle V_i,  V_j\rangle\ge \tau^p_d - c\frac{\max(\sqrt{np},\log n)\log n^{3/2}}{d}\Big]$ has $H$ as an edge-subgraph. $G$ is a sample from $\RGG(n,\dsphere, p(1 + \delta)),$ where 
$$p\delta = \prob_{X\sim \sphericaloned}\Big[X\in \big[\tau^p_d-c\tfrac{\max(\sqrt{np},\log n)\log n^{3/2}}{d}, \tau^p_d\big]\Big].$$
By Part D) of \cref{lemma:sphericalcdf}, $\delta = O\Big(\frac{\max(\sqrt{np},\log n)\log n^{3/2}{\log(1/p)^{1/2}}}{\sqrt{d}}\Big) = o(1).$  

Furthermore, as  $\indicator\Big[\langle V_i,  V_j\rangle\ge \tau^p_d + c\frac{\max(\sqrt{np},\log n)\log n^{3/2}}{d}\Big]\le H_{i,j}$ for all $i,j,$ it is the case that $G\setminus H$ is a set of edges which belongs to the edge-set defined by
\begin{align}
\label{align:fragileedges}
\Big\{(i,j)\; : \; \langle V_i,  V_j\rangle \in \big[\tau^p_d - c\tfrac{\max(\sqrt{np},\log n)\log n^{3/2}}{d},
\tau^p_d + c\tfrac{\max(\sqrt{np},\log n)\log n^{3/2}}{d}\big]\Big\}.\end{align}
By Part D) of \cref{lemma:sphericalcdf}, 
$q\coloneqq \prob\Big[\langle V_i,  V_j\rangle \in \big[\tau^p_d - c\frac{\max(\sqrt{np},\log n)\log n^{3/2}}{d},
\tau^p_d + c\frac{\max(\sqrt{np},\log n)\log n^{3/2}}{d}\big]\Big] = O\Big(p\frac{\max(\sqrt{np},\log n)\log n^{3/2}{\log(1/p)^{1/2}}}{\sqrt{d}}\Big) = o(p).$ Hence, in expectation, $G\setminus H$ is contained in a set of size $o(n^2p).$ 
Note that the events $\big\{\langle V_i,  V_j\rangle \in \big[\tau^p_d - c\frac{\max(\sqrt{np},\log n)\log n^{3/2}}{d},
\tau^p_d + c\frac{\max(\sqrt{np},\log n)\log n^{3/2}}{d}\big]\big\}_{i\neq j}$ are pairwise independent. Hence, by the second moment method, with high probability the set in \cref{align:fragileedges} is of cardinality at most $2\max(q, 1/n^{3/2})n^2 = o(pn^2).$ We reach the following conclusion.

\begin{corollary}
\label{cor:containmentforenumeration}
For $G\sim \RGG(n,\dsphere, p(1+\delta)), H\sim\ergraph,$ coupled as above, with high probability,
$
H \text{ is an edge-subgraph of $G$} \text{ and }
|E(G)| - |E(H)| \le \epsilon n^2p \text{ for some }\epsilon = o(1).
$
\end{corollary}

This is enough to conclude that $\RGG(n,\dsphere, p(1+\delta))$ has support of size $\exp(\Omega(n^2p\log(1/p))) = \exp(\Omega(nd\log n^{-7}))$ using the following two nearly trivial observations proved in \cref{appendix:graphcounting}. 

\begin{observation}[{\cite[Implicit in Theorem 7.5.]{bangachev2023random}}] 
\label{lemma:ersupport}
Suppose that $p \le 1/2, p = \Omega(1/n).$ Then, any set $A$ such that 
$
P_{H\sim \ergraph}[H \in A]\ge 1/2
$ satisfies $|A|\ge \exp(\frac{1}{6}n^2p\log(1/p)).$
\end{observation}

\begin{observation}
\label{lemma:downcontainmentsize}
Suppose that  $p \le 1/2, p = \Omega(1/n)$ and $\epsilon<1/100.$ Let $G$ be any graph on at most $n^2p$ edges. Then, the number of edge-subgraphs $H$ of $G$ such that $|E(G)| - |E(H)|\le \epsilon n^2p$ is at most $\exp(2n^2p\epsilon \log(1/\epsilon)).$
 \end{observation}

With probability more than $1/2,$ it is the case that 
$H$ is an edge-subgraph of $G$ by \cref{cor:containmentforenumeration} and $G$ has at most $2\binom{n}{2}p\le n^2p$ edges (again, edges of random geometric graphs are pairwise-independent). Let $B$ be the support of $\RGG(n,\dsphere, p(1+\delta)$ restricted to graphs on at most $n^2p$ edges.  Hence, the number of edge-subgraphs of graphs in $B$ which differ in at most $\epsilon n^2 p $ edges is at most $|B|\times \exp(2n^2p\epsilon \log(1/\epsilon)) = |B|\times \exp(o(n^2p\log(1/p)))$ by \cref{lemma:downcontainmentsize}. By \cref{lemma:ersupport}, 
$|B|\times \exp(o(n^2p\log(1/p)))\ge \exp(\frac{1}{6}n^2p\log(1/p)).$ Hence, 
$
|B| \ge  
\exp(\frac{1}{7}n^2p\log(1/p)) = 
\exp(\Omega(n^2d\log^{-7} n)).
$\end{proof}
\begin{remark}[On The Support Size of The RGG distribution]
\label{rmk:supportsizeofrgg}
The argument above shows that whenever $d\ge \max(np,\log^2n)\log(n)^3\log(1/p),$ the $\RGG$ distribution is supported on $\exp(\Omega(n^2p\log(1/p)))$ graphs. Whenever $d\le np\log(1/p)\log(n)^{-1},$ it is supported on $\exp(o(n^2p\log(1/p)))$ as the number of geometric graphs in dimension $d$ is $\exp(o(nd\log n))$ \cite{MCDIARMID2011627,SAUERMANN2021107593}. In particular, just like our spectral refutation algorithm shows (see \cref{rmk:couplingimpossibilityfromspectral}), this means that a coupling of the form $H_-\subseteq G\subseteq H_+$ cannot exist when $d\le np\log(1/p)\log(n)^{-1}.$
\end{remark}

%% file: SharpThresholds.tex
\section{Sharp Thresholds}
\label{sec:sharpthresholds}
\subsection{Proof of \texorpdfstring{\cref{thm:sharpthresholds}}{Sharp Thresholds}}
\label{sec:proofofsharptgresholds}
\begin{proof}[Proof of \cref{thm:sharpthresholds}]
Consider some monotone property $\mathcal{P}_n$ such that $\ergraph$ exhibits sharp thresholds for $\mathcal{P}_n.$ For any constant $\epsilon>0,$ let the critical window $\omega_\epsilon$ be of length $o_n(p).$ Let $f:[0,1]\longrightarrow[0,1]$ be the function defined by $f(p) = \prob_{\bfG\sim\ergraph}[\bfG\in \mathcal{P}].$ Let $p_n^c = f^{-1}(1/2).$ Similarly, let 
$g(p) = \prob_{\bfG\sim\RGGsphere}[\bfG\in \mathcal{P}].$ Let $q_n^c = g^{-1}(1/2).$ 

Take $\delta\in (0,0.1)$ to be any absolute constant.
Since $\ergraph$ exhibits sharp thresholds with respect to $\mathcal{P},$ $f^{-1}(\epsilon(1-\delta)) = p(1 + o_n(1)).$ Now, by \cref{thm:stochasticcoupling}, one can couple with high probability $1 - \psi$ for some $\psi = o(1) = o(\epsilon\delta),$ the graphs $\bfG_-\sim\RGG(n,\dsphere, f^{-1}(\epsilon(1-\delta))\times (1 - \frac{c\max(\sqrt{np^c_n},\sqrt{\log n})(\log^2 n)}{\sqrt{d}})$ and $\bfH_-\sim \mathsf{G}(n,f^{-1}(\epsilon(1-\delta))).$ Hence, as $\mathcal{P}_n$ is monotone, 
$$
\prob[\bfG_-\in \mathcal{P}_n]\le 
\prob[\bfG_-\not \subseteq \bfH_-] + 
\prob[\bfG_- \subseteq \bfH_-, \bfH_-\in \mathcal{P}_n]\le 
\psi + \epsilon(1-\delta)\le \epsilon.
$$
Hence, $g^{-1}(\epsilon)\ge f^{-1}(\epsilon(1-\delta))\times (1 - \frac{c\max(\sqrt{np^c_n},\sqrt{\log n})(\log^2 n)}{\sqrt{d}}) = f^{-1}(\epsilon(1-\delta))(1-o_n(1)).$ As this holds for any $\delta>0,$
$g^{-1}(\epsilon)\ge f^{-1}(\epsilon(1-o_n(1)))(1 - o_n(1)).$ Similarly, $g^{-1}(1-\epsilon)\le f^{-1}(1-\epsilon(1-o_n(1)))(1 + o_n(1)).$

Altogether, this means that $$g(1-\epsilon) - g(\epsilon) = o(f^{-1}(1-\epsilon(1-o_n(1)))) + o(f^{-1}(\epsilon(1-o_n(1)))) + 
f^{-1}(1-\epsilon(1-o_n(1))) - 
f^{-1}(\epsilon(1-o_n(1))) = o(p_n^c).
$$
Taking $\epsilon\longrightarrow 1/2$ in 
$$
f^{-1}((1-\epsilon)(1-o_n(1)))(1-o_n(1))\le
g^{-1}(1-\epsilon)\le f^{-1}(1-\epsilon(1-o_n(1)))(1 + o_n(1))
$$
also implies that $q_n^c = p_n^c(1 + o(1)).$
\end{proof}
\subsection{RGG and the FKG Lattice Condition}
A measure $\mu$ over $\{0,1\}^N$ is strictly positive if $\mu(\omega)>0$ for each $\omega\in \{0,1\}^n.$ A strictly positive measure satisfies the FKG lattice condition if for any two $\omega_1,\omega_2\in \{0,1\}^N,$ it holds that $\mu(\omega_1\vee \omega_2)\mu(\omega_1\wedge\omega_2)\ge \mu(\omega_1)\mu(\omega_2)$ \cite{grimmett06fkg}. We denote pointwise maximum with $\vee$ and pointwise minimum with $\wedge.$

\begin{proposition}
 \label{prop:failureofFKG}
 Suppose that $d\ge 3.$ Then, the measure corresponding to $\RGG(3,\dsphere,1/2)$ is strictly positive but does not satisfy the FKG lattice property. 
\end{proposition}
\begin{proof} Denote by $\mu$ the measure corresponding to $\RGG(3,\dsphere,1/2).$ Due to symmetry, note that $\mu(\omega)$ only depends on $|\omega|.$ For simplicity of notation, we will also write $\mu(|\omega|)$ for $\mu(\omega).$ 

We will prove that for some $a\in (0,1/8),$ it is the case that 
\begin{align}
\label{eq:muvalues}
\mu(3) = 1/8 + a, \mu(2) = 1/8 - a, \mu(1) = 1/8 + a, \mu(0) = 1/8-a.
\end{align}
This is enough to show that the FKG lattice condition does not hold. Indeed, take $\omega_1 = \{(1,2)\}, \omega_2 = \{(2,3)\}.$ Then, 
$\omega_1\vee \omega_2 = \{(1,2), (2,3)\},\omega_1\wedge \omega_2 =\emptyset.$ For the FKG lattice property to hold, it must be the case that $\mu(2)\mu(0)\ge \mu(1)^2.$ This, however, is false as the left hand-side is $(1/8 - a)^2$ while the  right hand-side is $(1/8 +a)^2.$

Going back to \cref{eq:muvalues}. Let $\mu(3) = 1/8 + a,$ where $a$ is for now some number that we will later prove to be strictly positive. 
First, note that 
$$
\mu(3) + \mu(2) = 
\prob[\bfG_{12} = 1,\bfG_{13} = 1, \bfG_{23} = 1] + 
\prob[\bfG_{12} = 1,\bfG_{13} = 1, \bfG_{23} = 0] = 
\prob[\bfG_{12} = 1,\bfG_{13} = 1].
$$
The last expression equals $1/4$ as every two edges in $\RGGsphere$ are pairwise independent. Hence, $\mu(2) = 1/4 - \mu(3) = 1/8 - a.$ Similarly, $\mu(1) = 1/4-\mu(2) = 1/8 + a$ as 
$$
\mu(2) + \mu(1) = 
\prob[\bfG_{12} = 1,\bfG_{13} = 1, \bfG_{23} = 0] + 
\prob[\bfG_{12} = 1,\bfG_{13} = 0, \bfG_{23} = 0] = 
\prob[\bfG_{12} = 1,\bfG_{23} = 0 ],
$$
and $\mu(0) = 1/4-\mu(1) = 1/8 - a$ as 
$$
\mu(1) + \mu(0) = 
\prob[\bfG_{12} = 1,\bfG_{13} = 0, \bfG_{23} = 0] + 
\prob[\bfG_{12} = 0,\bfG_{13} = 0, \bfG_{23} = 0] = 
\prob[\bfG_{13} = 0,\bfG_{23} = 0].
$$

Thus, all that is left to show is that $a\ge 0.$
We accomplish this by bounding $\mu(3).$ 
\begin{align*}
    & \mu(3) = \prob[\bfG_{12} = 1,\bfG_{13} = 1, \bfG_{23} = 1]\\
    & = \prob[\bfG_{23} = 1|\bfG_{12} = 1, \bfG_{13} = 1]\times \prob[\bfG_{12} = 1, \bfG_{13} = 1]\\
    & = \frac{1}{4}
    \prob[\bfG_{23} = 1|\bfG_{12} = 1, \bfG_{13} = 1]\\
    & = \frac{1}{4}\prob[\langle V_2, V_3\rangle\ge 0|\langle V_1, V_2\rangle\ge 0, \langle V_1, V_3\rangle\ge 0].
\end{align*}
First, note that this probability is the same if $V_1, V_2, V_3$ were isotropic Gaussian vectors instead of spherical - the sign of inner products depends only on the direction but not on the magnitude of vectors. 
Due to rotational invariance, we can assume that $V_1 \in \mathsf{span}\{e_1\},$ so $V_1 =z_{1,1} e_1,$ where $z_{1,1} >0,$ and $V_2\in \mathsf{span}\{e_1,e_2\}.$ So, $V_2 = z_{2,1} e_1 + z_{2,2} e_2.$ Finally, let $V_3 = z_{3,1}e_1 + z_{3,2} e_2 + z_{3,3}e_3.$ Since the components of an isotropic Gaussian are are independent in different directions, all of the numbers $z_{1,1}, z_{2,1},z_{2,2}, z_{3,1}, z_{3,2}, z_{3,3}$ are independent, Furthermore,  $z_{2,1}, z_{3,1}, z_{3,2}$ are standard Gaussians and
$z_{1,1}, z_{2,2}, z_{3,3}> 0$ as these are norms of certain projections of the Gaussian vectors.
We want to show that 
$$
\prob[z_{3,1}z_{2,1} + z_{3,2}z_{2,2}\ge 0 |z_{3,1}\ge 0, z_{2,1}\ge 0]\in (1/2, 1).
$$
We compute
\begin{align*}
& \prob[z_{3,1}z_{2,1} + z_{3,2}z_{2,2}\ge 0 |z_{3,1}\ge 0, z_{2,1}\ge 0]\\
& = 
\int_{0}^{+\infty}
\prob[z_{3,1}z_{2,1} + z_{3,2}z_{2,2}\ge 0 |z_{3,1}\ge 0, z_{2,1}\ge 0, z_{3,1}z_{2,1} = x]\times \prob[z_{3,1}z_{2,1} = x| z_{3,1}\ge 0, z_{2,1}\ge 0]dx\\
& = 
\int_{0}^{+\infty}
\prob[z_{3,2}\ge -x/z_{2,2}]\times \prob[z_{3,1}z_{2,1} = x| z_{3,1}\ge 0, z_{2,1}\ge 0]dx.
\end{align*}
Note that for each $x,z_{2,2}\ge 0,$ it holds that 
$\prob[z_{3,2}\ge -x/z_{2,2}]\in (1/2,1)$ as $z_{3,2}$ is a standard Gaussian and $-x/z_{2,2}<0.$ The conclusion follows.
\end{proof}




%% file: Discussion.tex
\section{Discussion and Open Problems}
\label{sec:discussion}
\subsection{Robust Testing With Fewer Corruptions}
One of the main results of the current paper is settling (up to log factors) the regime in which testing between $\ergraph$ and $\RGGsphere$ with $\epsilon \binom{n}{2}p$ corruptions is possible when $\epsilon >0$ is a small constant. We show via \cref{thm:algorithmiccoupling,thm:robustestingsos} that the threshold is $d = \widetilde{\Theta}(np)$ both computationally and information-theoretically. The analogous question without corruptions, i.e.  $\epsilon = 0,$ is still open as discussed in the introduction. It is natural to interpolate between $\epsilon = 0$ and $\epsilon = \Theta(1).$

\begin{problem}
\label{problem:robusttestinglittleo1}
Given are $n,d\in \mathbb{N},$ $p\in [0,1]$ and a function $\epsilon(n,p)\ge 0$ such that $\epsilon(n,p) \le K_r$ for some universal constant $K_r.$
Consider the following hypothesis testing problem. One observes a graph $H$ and needs to decide between $H_0: H = \mathcal{A}(\bfG),\bfG\sim \ergraph$ and 
$H_1: H = \mathcal{A}(\bfG),\bfG\sim \mathsf{RGG}(n,d,p).$ Here, $\mathcal{A}:\{0,1\}^{\binom{n}{2}}\longrightarrow \{0,1\}^{\binom{n}{2}}$ is any unknown mapping such that 
$\mathcal{A}(G)$ and $G$ differ in at most $\epsilon(n,p)\times  \binom{n}{2}p$ entries for any adjacency matrix $G\in \{0,1\}^{\binom{n}{2}}.$
\end{problem}

We expect that this question will require both novel algorithmic and lower-bound ideas. Neither the spectral refutation algorithm (when $\epsilon = \Theta(1)$) nor the signed triangle count algorithm (when $\epsilon = 0$) seems to generalize trivially. Regarding lower-bounds, even the case of no corruptions is still open. Particularly interesting is whether there is a statistical-computational gap for any $\epsilon(n,p).$ There is no evidence for such a gap when $\epsilon(n,p) = 0$ and $\epsilon(n,p) = \Theta(1).$ 

\cref{thm:algorithmiccoupling,thm:stochasticcoupling} show that if $\epsilon(n,p) \ge C\frac{\max(\sqrt{np}, \sqrt{\log n})(\log n)^{2}}{\sqrt{d}}$ for some absolute constant $C,$ the distinguishing problem is impossible information theoretically against the (efficient) 
adversary from \cref{thm:algorithmiccoupling}. We do not expect this to be tight for all values of $\epsilon$ (in particular, when $\epsilon$ approaches 0). This naturally also raises a question about the optimality of \cref{thm:algorithmiccoupling,thm:stochasticcoupling}. We only phrase the question for the latter. 

\begin{problem}
\label{problem:optimalerrorcoupling}
What is the smallest value of $\delta(d,n,p)$ for which there exists a coupling of $G\sim\RGG(n,\dsphere,p),$ $H_-\sim \mathsf{G}(n,p(1-\delta(d,n,p)))$ and $ H_+\sim \mathsf{G}(n,p(1+\delta(d,n,p))),$ 
 such that
$
H_-\subseteq G \subseteq H_+
$
with high-probability.
\end{problem}

\subsection{Embedding Random Geometric Graphs}
\paragraph{New Algorithms For Embedding.} \cref{thm:algorithmiccoupling} can also be viewed as an approximate embedding of a sample $H$ from $\ergraph$ as a geometric graph. This naturally raises the question whether one can use \cref{eq:couplingalgorithm} to embed a sample $H$ from $\RGGsphere$ (with the true vectors remaining unknown). Our martingale-based analysis fails in this scenario as the independence of variables $H_{ji}$ is crucial. Nevertheless, the fact that the marginal distribution of vectors $V_1, V_2, \ldots, V_n$ is uniform and independent makes such an algorithm extremely appealing as it gives a certain sampling flavor of the resulting embedding: It produces a uniform sample from a certain set of $n$-tuples of latent vectors which nearly match the input adjacency matrix.  

\paragraph{Robust Embedding.}
A different algorithmic question is whether there exist embedding algorithms in the setting of adversarial corruptions. One way to phrase this question (using a Frobenius norm error) is:

\begin{problem} Let $V = (V_1, V_2, \ldots, V_n)\in \mathbb{R}^{d\times n}$ where $V_1, V_2, \ldots, V_n\iidsim\unif(\dsphere) $ and $G$ be the corresponding geometric graph with $\epsilon \binom{n}{2}p$ edge corruptions
(or $\epsilon n$ node corruptions). On input $G,$ find a matrix $\widehat{W}$ such that 
$
\|\widehat{W} - V^TV\|_F = \|V^TV\|_F\times (f(\epsilon) + o(1)),
$
where $f(\epsilon)$ is some function such that $\lim_{\epsilon\longrightarrow 0}f(\epsilon) = 0.$
\end{problem}
The case of no corruptions (i.e. $\epsilon = 0$) is treated in \cite{li2023spectral} via a spectral algorithm which succeeds whenever $d = o(np).$


\subsection{A Reverse Coupling?} \cref{eq:couplingalgorithm} efficiently couples with high probability a sample $H$ from $\ergraph$ and a sample $G$ from $\RGGsphere$ on input $H$ such that they differ on $o(n^2p)$ edges. One may naturally ask the question of whether such an efficient coupling algorithm exists on input the random geometric graph. 

\begin{problem} Design an efficient algorithm which on input $G\sim \RGGsphere$ where $d \ge np\times \log^C n$ for an appropriate constant $C$ produces a graph $H(G)$ such that:
\begin{enumerate}
    \item $\TV(\law(H(G)), \ergraph) = o(1).$
    \item With high probability, $H(G)$ and $G$ differ on $o(n^2p)$ edges.
\end{enumerate}
\end{problem}

The question is only about efficiency as one can simply reverse the coupling in \cref{thm:algorithmiccoupling}. 

\subsection{Other Random Geometric Graph Distributions}
We showed that the central result of the current paper
\cref{thm:algorithmiccoupling} has a wide range of applications, including to robust statistics, enumeration, and sharp thresholds. The proof technique, however, is specialized to spherical random geometric graphs. Naturally, one may want a more systematic approach which also works for other high-dimensional random geometric graphs such as over the torus \cite{bangachev2023detection} or Boolean hypercube \cite{bangachev2023random} or non-uniform distributions such as the anisotropic random geometric graphs in \cite{Mikulincer20,Brennan22AnisotropicRGG}.
We believe that extending our results and techniques to other geometries is an exciting open direction. Our only progress in that direction is for Gaussian random geometric graphs, which is nearly identical to the spherical case.

\begin{remark}
\label{rmk:sphericaltogaussian}
\cref{thm:stochasticcoupling,thm:algorithmiccoupling} hold with the exact same parameters (up to multiplicative constants) for Gaussian RGG. A sample from $\RGGgauss$ is an $n$-vertex graph $G$ defined by sampling $Z_1, Z_2, \ldots, Z_n\iidsim\mathcal{N}(0,\frac{1}{d}I_d)$ and setting $G_{ji} = \indicator[\langle Z_i, Z_j\rangle\ge \xi^p_d]$ where $\xi^p_d$ is such that the expected density is $p.$ One coupling algorithm is \cref{eq:couplingalgorithm} except that at the end one defines
$
Z^{(i-1)}_i = \chi_i \times V^{(i-1)}_i,
$
where the variables $(\chi_i)_{i= 1}^n$ are independent of everything else and are distributed as $\|\mathcal{N}(0,\frac{1}{d}I_d)\|_2.$ Now, use that $|\chi_i - 1|\le \frac{\log n}{\sqrt{d}}$ with high probability (e.g. \cite{bangachev2024fourier}) and the fact that $|\tau^p_d - \xi^p_d| = \widetilde{O}(1/d)$ \cite{bangachev2024fourier}. One immediately concludes that when $d\gg np,$ the robust testing \cref{problem:robusttesting} is information theoretically impossible. This is perhaps surprising in the sparse case due to the difference between spherical and Gaussian RGG \cite{bangachev2024fourier}. There is an efficient test that distinguishes $\RGGgauss$ and $\ergraph$ when $d\gg n^{3/2}p, p = O(n^{-3/4})$ \cite{bangachev2024fourier} in the case of no corruptions, a regime in which  $\RGGsphere$ and $\ergraph$ are conjectured to be information-theoretically indistinguishable. Yet, when testing with adversarial corruptions no such difference between $\RGGgauss$ and $\RGGsphere$ seems to emerge.
\end{remark}

%% file: SphericalDistribution.tex
\section{Proof of \texorpdfstring{\cref{lemma:sphericalcdf}} {Spherical Properties}}
\label{appendix:proofofspherical}
We will use the fact that the density of $\sphericaloned$ is $f_d(x)\coloneqq \frac{\Gamma(d/2)}{\Gamma(d/2 - 1/2)\sqrt{\pi}}(1-x^2)^{\frac{d-3}{2}}$ and, for some absolute constant $C_\Gamma,$ 
$\frac{\sqrt{d}}{C_\Gamma}\le\frac{\Gamma(d/2)}{\Gamma(d/2 - 1/2)\sqrt{\pi}}\le C_\Gamma\sqrt{d}$ \cite{Bubeck14RGG}. Part A) appears explicitly in \cite{Bubeck14RGG}.

\subsection*{Part B)}
Note that $\expect_{X\sim \sphericaloned}[X\; | \; X\ge \tau^p_d]\ge \tau^p_d.$ Hence, if $p \le 1/4$ there is nothing to prove by Part A). If $p \in (1/4, 1/2],$ we compute:
\begin{align*}
    & \expect_{X\sim \sphericaloned}[X\; | \; X\ge \tau^p_d]\\
    &\ge \expect_{X\sim \sphericaloned}[X\; | \; X\ge 0]\\
    & = 2\int_0^1 \frac{\Gamma(d/2)}{\Gamma(d/2 - 1/2)\sqrt{\pi}}x(1-x^2)^{\frac{d-3}{2}}dx\\
    & \ge 2\frac{\Gamma(d/2)}{\Gamma(d/2 - 1/2)\sqrt{\pi}}
    \int_{1/\sqrt{d}}^{2/\sqrt{d}}x(1-x^2)^{\frac{d-3}{2}}dx\\
    & \ge \frac{2\sqrt{d}}{C_\Gamma}
    \int_{1/\sqrt{d}}^{2/\sqrt{d}}\frac{1}{\sqrt{d}}(1-\frac{4}{d})^{\frac{d-3}{2}}dx 
     \ge \frac{2\sqrt{d}}{C_\Gamma}\times \frac{1}{\sqrt{d}}\frac{1}{\sqrt{d}}e^{-4}\ge \frac{C'\sqrt{\log 2}}{\sqrt{d}}\ge \frac{\sqrt{\log (1/p)}}{C_B\sqrt{d}}
\end{align*}
for some absolute constants $C', C_B.$
\subsection*{Part C)}
We will prove that $\expect_{X\sim \sphericaloned}[X^k\; |\; X\ge \tau^p_d]\le c_k(\tau^p_d+c_k)^k$ for some absolute constant $c_k$ depending only on $k.$ Clearly, this is enough. Again, we begin with the explicit formula for the expectation:
\begin{align*}
    & \expect_{X\sim \sphericaloned}[X^k\; | \; X\ge \tau^p_d]= \frac{\int_{\tau^p_d}^1x^kf_d(x) dx}{\int_{\tau^p_d}^1f_d(x) dx}.
\end{align*}
We now analyse the numerator in several steps. Let $\beta_k>0$ be a constant, depending only on $k,$ that we will chose later.

\paragraph{Small values of $x.$} We compute 
\begin{align*}
    & \int_{\tau^p_d}^{\tau^p_d + \frac{\beta_k}{\sqrt{d}}}x^kf_d(x) dx\\
    & \le \int_{\tau^p_d}^{\tau^p_d + \frac{\beta_k}{\sqrt{d}}}(\tau^p_d + \frac{\beta_k}{\sqrt{d}})^kf_d(x) dx\\
    & \le \int_{\tau^p_d}^{1}(\tau^p_d + \frac{\beta_k}{\sqrt{d}})^kf_d(x) dx\\
    & \le (\tau^p_d + \frac{\beta_k}{\sqrt{d}})^k
    \int_{\tau^p_d}^{1}f_d(x) dx,
\end{align*}
so this part of the numerator is of the desired order.

\paragraph{Large Values of $x.$} The key argument when 
$\frac{1}{2}\ge x\ge \tau^p_d + \frac{\beta_k}{\sqrt{d}}$ is that (for an appropriately chosen $\beta_k$), one has 
\begin{align}
\label{eq:maininequalitylargex}
(\tau^p_d + \frac{2\beta_k}{\sqrt{d}})^kf_d(x - \frac{\beta_k}{\sqrt{d}})\ge
x^kf_d(x).
\end{align}
First, note that this is enough as it will imply that
$$
\int_{\tau^p_d + \frac{\beta_k}{\sqrt{d}}}^{1/2}x^kf_d(x)\; dx\le 
\int_{\tau^p_d + \frac{\beta_k}{\sqrt{d}}}^{1/2}
(\tau^p_d + \frac{2\beta_k}{\sqrt{d}})^kf_d(x - \frac{\beta_k}{\sqrt{d}})\; dx\le 
(\tau^p_d + \frac{2\beta_k}{\sqrt{d}})^k
\int_{\tau^p_d}^1f_d(x)\; dx.
$$
Going back to \cref{eq:maininequalitylargex}. 
As $f_d(x)$ is decreasing, the inequality is trivial for $x\le \tau^p_d + \frac{2\beta_k}{\sqrt{d}}.$ When $x>\tau^p_d + \frac{2\beta_k}{\sqrt{d}},$ set $y   = x - \frac{\beta_k}{\sqrt{d}}\ge \frac{\beta_k}{\sqrt{d}}.$ Using the explicit formula for $f_d(x),$ we need to show that 
\begin{align*}
& (\tau^p_d + \frac{2\beta_k}{\sqrt{d}})^k\ge 
(y + \frac{\beta_k}{\sqrt{d}})^k \Big(\frac{1 - (y + \frac{\beta_k}{\sqrt{d}})^2}{1-y^2}\Big)^{(d-3)/2}\Longleftrightarrow\\
& (\tau^p_d + \frac{2\beta_k}{\sqrt{d}})^k\ge
(y + \frac{\beta_k}{\sqrt{d}})^k
\Big(1 - \frac{2y\frac{\beta_k}{\sqrt{d}} + y^2}{1-x^2}\Big)^{(d-3)/2}\Longleftrightarrow\\
& (\tau^p_d + \frac{2\beta_k}{\sqrt{d}})^k\ge 
(y + \frac{\beta_k}{\sqrt{d}})^k
\Big(1 - \frac{y\frac{\beta_k}{\sqrt{d}} + y^2}{3/4}\Big)^{(d-3)/2} \Longleftrightarrow\\
& (\tau^p_d + \frac{2\beta_k}{\sqrt{d}})^k\ge 
(y + \frac{\beta_k}{\sqrt{d}})^k
\exp( -\alpha y(y + \frac{\beta_k}{\sqrt{d}})d)
\end{align*}
for some absolute constant $\alpha$ since $y\frac{\beta_k}{\sqrt{d}} + y^2\le /\frac{1}{2}\frac{\beta_k}{\sqrt{d}} + 1/4\le 2/3<3/4.$ Now, the function $u(z) = z^k\exp(-\alpha dz(z +\frac{\beta_k}{\sqrt{d}}))$ is decreasing on $[\frac{\beta_k}{\sqrt{d}}, 1]$ for a sufficiently large $\beta_k.$ Indeed, note that its derivative is 
$$
\exp(-\alpha dz(z +\frac{\beta_k}{\sqrt{d}}))
(kz^{k-1} - z^k\alpha d(2z + \frac{\beta_k}{\sqrt{d}})).
$$
For large enough $\beta_k,$ we have that $kz^{k-1}\le z^k\alpha \beta_k\sqrt{d}$ for all $z\ge \beta_k/\sqrt{d},$ from which the fact that the derivative is negative follows. This is enough as 
$$
(y + \frac{\beta_k}{\sqrt{d}})^k
\exp( -\alpha y(y + \frac{\beta_k}{\sqrt{d}})d)\le 
(\tau^p_d + \frac{2\beta_k}{\sqrt{d}})^k
\exp( -\alpha (\tau^p_d + \frac{\beta_k}{\sqrt{d}})(\tau^p_d + \frac{2\beta_k}{\sqrt{d}})d)\le 
(\tau^p_d + \frac{2\beta_k}{\sqrt{d}})^k
$$
in that case.

\paragraph{Extremely large values of $x.$} Note that as $d = \omega(\log(1/p)),$ when $x\in [1/2,1],$ $f_d(x)\le (3/4)^{(d-3)/2} = o(\sqrt{\log(1/p)/d}).$

\subsection*{Part D)}
The main observation is the following. 

\begin{observation} If $u \le C (\log(1/p)^{-1/2}/\sqrt{d}),$ for some absolute constant $C,$ then
$$
\frac{1}{2}\le 
f_d(\tau^p_d + u)/f_d(\tau^p_d)
\le 2.
$$
\end{observation}
\begin{proof} We write
  \begin{align*}
      & f_d(\tau^p_d + u)/f_d(\tau^p_d) = \Bigg(\frac{1 -(\tau^p_d + u)^2}{1 - (\tau^p_d)^2}\Bigg)^{(d-3)/2} = 
      \Bigg(1 - \frac{2\tau^p_d u + u^2}{1 - (\tau^p_d)^2}\Bigg)^{(d-3)/2}.
  \end{align*}  
Whenever $d = \omega(\log 1/p),$ we have 
$1 - (\tau^p_d)^2\ge 1/2$ by part A). Whenever $u \le  C(\log(1/p)^{-1/2}/\sqrt{d}),$ clearly $2\tau^p_d u + u^2 \le 2C/d.$ Hence, the expression above is in $(1 \pm 2C/d)^{(d-3)/2}\in [1/2,2].$
\end{proof}

Now, in particular, this means that $f_d(\tau^p_d) \le C'p\sqrt{d}\sqrt{\log 1/p}$ for some absolute constant $C'.$ Indeed, notice that 
$$
p = \int_{\tau^p_d}^1f_d(x)dx\ge 
\int_{\tau^p_d}^{\tau^p_d + C (\log(1/p)^{-1/2}/\sqrt{d})}f(x)dx\ge 
\int_{\tau^p_d}^{\tau^p_d + C (\log(1/p)^{-1/2}/\sqrt{d})}
f(\tau^p_d)/2dx.
$$
Hence, $f_d(x)\le 2C'p\sqrt{d}\sqrt{\log 1/p}$ for all $x\in [\tau^p_d - \Delta, \tau^p_d + \Delta]$ for sufficiently small $\Delta\le  C (\log(1/p)^{-1/2}/\sqrt{d})$
and the claim follows.

%% file: GraphCountingAppendix.tex
\section{Simple Lemmas on Graph Counting}
\label{appendix:graphcounting}

\begin{proof}[Proof of \cref{lemma:ersupport}] This lemma appears in \cite{bangachev2023random}, but we give the proof for completeness. Consider $H \sim \ergraph.$ First, as edges are independent, by Chenyshev's inequality, 
$$\prob\Big[|E(H)|\le \frac{1}{2}\binom{n}{2}p\Big] = \prob\Big[|E(H)|\le \frac{1}{2}\expect[|E(H)|]\Big]  = o(1).
$$
On the other hand, for any fixed graph $K$ on at least $\frac{1}{2}\expect|E(H)| = \frac{1}{2}\binom{n}{2}p$ vertices, we have 
$$
\prob[H = K] = 
p^{|E(K)|}
(1 - p)^{\binom{n}{2} - |E(K)|}\le 
p^{|E(K)|}\le 
\exp\Big(\log p \frac{1}{2}\binom{n}{2}p\Big)\le 
\exp\Big(-\frac{1}{5}n^2p\log \frac{1}{p}\Big).
$$
Hence, a $1/2 - o(1)$ mass of the set $A$ must be supported on graphs with mass $\exp\big(-\frac{1}{5}n^2p\log \frac{1}{p}\big)$ each with respect to $\ergraph.$ The statement follows.
\end{proof}

\begin{proof}[Proof of \cref{lemma:downcontainmentsize}]
The number of such subgraphs of $H$ is simply 
$
\sum_{r = 0}^{\epsilon n^2p}\binom{|E(H)|}{r}.
$ We can trivially bound 
\begin{align*}
     \sum_{r = 0}^{\epsilon n^2p}\binom{|E(H)|}{r}
   & \le  \sum_{r = 0}^{\epsilon n^2p}\binom{n^2p}{r}
    \le \epsilon n^2 p \binom{n^2p}{\epsilon n^2 p }\\
    &\le \epsilon n^2 p\Big(\frac{en^2p}{\epsilon n^2 p }\Big)^{\epsilon n^2 p }
     = \epsilon n^2 p (e/\epsilon)^{\epsilon n^2 p} \le 
    \exp(2n^2p\epsilon \log(1/\epsilon)).\qedhere
\end{align*}
\end{proof}

%% file: RGGRepresentation.tex
\section{Random Geometric Graphs as Recursively Planted \ER}
\label{appendix:representation}
Here, we discuss in more detail \cref{obs:rggasrecursiveER}.

\paragraph{More General Flip-Orientation Maps}
In analogy to \cref{sec:fliporientation}, we can define a more general flip-orientation map. Let $\phi(\cdot|[u,v]):[u,v]\longrightarrow [u,v]$ for $u<\tau^p_d<v$ be the unique monotone involution such that $\phi(u|[u,v]) = v, \phi(v|[u,v]) = u,\phi(\tau^p_d|[u,v]) = \tau^p_d$ and $\phi(\cdot|[u,v]):[u,\tau^p_d]\longrightarrow [\tau^p_d,v]$ and 
$\phi(\cdot|[u,v]):[\tau^p_d,v]\longrightarrow [u,\tau^p_d]$ are measure-preserving with respect to $\sphericaloned|_{[u,\tau^p_d]},\sphericaloned|_{[\tau^p_d,v]}.$  In particular, $\phi(\cdot) = \phi(\cdot |[-1,1])$ in the notation of \cref{sec:fliporientation}.

Define also 
\begin{equation}
    \begin{split}
        & \rho^b_a(z|[u,v])\coloneqq
        \begin{cases}
        z\quad \text{ if } \langle z,a\rangle\not \in [u,v],\\
        z\quad \text{ if }\indicator[\langle z,a\rangle \ge \tau^p_d] = b,\langle z,a\rangle \in [u,v],\\
        (z - \langle z,a\rangle a)\sqrt{\frac{1 - \phi^2(\langle z,a\rangle|[u,v])}{1 - \langle z,a\rangle^2}} + 
        \phi(\langle a,z\rangle|[u,v]]) a \quad\text{ if }\indicator[\langle z,a\rangle \ge \tau^p_d] \neq b, \langle z,a\rangle \in [u,v].
        \end{cases}
    \end{split}
\end{equation}

Define $s^b(\langle a,z\rangle |[u,v]]) = s^b(\langle a,z\rangle)\times \indicator[\langle a,z\rangle \in [u,v]].$ 
Also, define analogously $\psi(\cdot | [u,v]), \kappa(\cdot | [u,v]).$
Let $q_{[u,v]} =\frac{\prob_{X\sim\sphericaloned}[X\in [\tau^p_d, v]]}{ 
\prob_{X\sim\sphericaloned}[X\in [u, v]]
}.$
If $b\sim \Bernoulli(q_{[u,v]}),$ 
 then the equivalent statements of \cref{obs:unifromfixeddirection} and \cref {obs:signexpectation} hold with same argument.  

\begin{observation} 
\label{obs:unifromfixeddirectiongeneral}
If $b\sim\Bernoulli(q_{[u,v]}), z\sim \unif(\dsphere)$ and $a,b, z$ are independent, then\linebreak  $\rho^b_a(z|[u,v]])\sim \unif(\dsphere).$
\end{observation}

\begin{observation}  
\label{obs:signexpectationgeneral}
If $b\sim\Bernoulli(q_{[u,v]}), z\sim \unif(\dsphere)$ and $a,b, z$ are independent, then \linebreak $\expect[s^b(\langle z, a\rangle | [u,v])]=q_{[u,v]}\times \prob_{X\sim\sphericaloned}[X\in [u,v]].$
\end{observation}

Finally, in addition to \cref{prop:elemntarybounds}, one also has the following trivial bounds. 

\begin{proposition}
\label{prop:elemntaryboundsgeneral}
Suppose that 
$|u -\tau^p_d|, |v - \tau^p_d| = o(1).$ Then, for some absolute constant $C:$
\begin{enumerate}
    \item $\big|\phi(\langle a, z \rangle|[u,v]) - \langle a, z \rangle\big|\times \indicator[\langle a,z\rangle \in [u,v]]\le (v-u), $
    \item $\big|\psi(\langle a, z\rangle | [u,v])\big|\times \indicator[\langle a,z\rangle \in [u,v]]\le C(v-u)^2,$
    \item $\big|\kappa(\langle u, v\rangle)\big|\times \indicator[\langle a,z\rangle \in [u,v]]\le C(v-u).$
\end{enumerate}
\end{proposition}

\paragraph{Defect Edges After \cref{eq:couplingalgorithm}}
Observe that one can choose constants $c_1,c_2>c$ such that 
\begin{align*}
& \frac{1}{2} = \prob\Big[\langle \Vfin{i}, \Vfin{j}\rangle\le \tau^p_d\quad \Big|
 \tau^p_d -\frac{c_1\max(\sqrt{np}, \sqrt{\log n})(\log n)^{3/2}}{d}
\le \langle  \Vfin{i}, \Vfin{j}\rangle\le \tau^p_d +\frac{c_2\max(\sqrt{np}, \sqrt{\log n})(\log n)^{3/2}}{d}\Big]. 
\end{align*}
Hence, one can attempt to fix the ``defect edges'' by using fresh Bernoulli samples as follows.

\paragraph{Multiple Rounds of Edge Flipping.}
\begin{enumerate}
    \item Draw independently $H^0\sim \ergraph, H^1\sim \mathsf{G}(n,1/2), H^2\sim\mathsf{G}(n,1/2), \ldots, H^T\sim\mathsf{G}(n,1/2)$ (where $T = O(\frac{\log d}{\log \kappa} + \frac{\log n}{\log d})$ to be specified later).
    \item Draw $V^{(0)}_1(0), V^{(0)}_2(0), \ldots, V^{(0)}_n(0)\iidsim\unif(\dsphere).$
    \item On input $H^0, V^{(0)}_1(0), V^{(0)}_2(0), \ldots, V^{(0)}_n(0)$ apply \cref{eq:couplingalgorithm} to produce $V^{(0)}_1(1), V^{(0)}_2(1), \ldots, V^{(0)}_n(1)$ (where $V^{(0)}_i(1) = V^{(i-1)}_i(0) = V^{(i-1)}_i$ in the notation of \cref{eq:couplingalgorithm}). Set $U_1 = \frac{c_1\max(\sqrt{np}, \sqrt{\log n})(\log n)^{3/2}}{d},$ $ L_1 = \frac{c_2\max(\sqrt{np}, \sqrt{\log n})(\log n)^{3/2}}{d}$ for some $c_1, c_2>c$ so that $\prob_{X\sim\sphericaloned}[X\le \tau^p_d|\tau^p_d - L_1\le X\le \tau^p_d + U_1] = \frac{1}{2}.$
    \item For $t = 1,2,\ldots, T:$
    \begin{enumerate}
        \item Initialize $F_t = \emptyset.$ 
        \item Go in lexicographic order of the pairs $(j,i):$
        \begin{enumerate}
            \item If $\langle V^{(i-1)}_j(t), V^{(i-1)}_i(t)\rangle  \in [\tau^p_d - L_t, \tau^p_d + U_t],$ 
        add $(ji)$ to $F_t.$
        \item Update $V^{(i)}_j(t) = \rho^{H^t_{ji}}_{V^{(i-1)}_i(t)}(V^{(i-1)}_j(t)| [\tau^p_d - L_t, \tau^p_d + U_t]).$
        \end{enumerate}
        \item Set $V^{(0)}_{i}(t+1) = V^{(i-1)}_i(t)$ for all $i.$ For some large absolute constant $C,$ set 
        \begin{align}
        \label{eq:choiceofsymmetricintervals}
         U_{t+1}= C(L_t + U_t)\max\Big(
        \frac{\log^{3/4} n\sqrt{np(L_t + U_t)}}{d^{1/4}}
        ,\frac{\log n}{\sqrt{d}}\Big)\text{ and }L_{t+1}
        \end{align}
        such that $\prob_{X\sim\sphericaloned}[X\le \tau^p_d|\tau^p_d - L_{t+1}\le X\le \tau^p_d + U_{t+1}] = 1/2.$
    \end{enumerate}
    \item Output $V^{(0)}_1(T), V^{(0)}_2(T), \ldots, V^{(0)}_n(T)$ and the associated random geometric graph $G^T$ where $(G^T)_{ji} = \indicator[\langle V^{(0)}_j(T), V^{(0)}_i(T)\rangle \ge \tau^p_d].$
\end{enumerate}
One can reason similarly as in \cref{sec:coupling} using \cref{obs:signexpectationgeneral} and 
\cref{obs:unifromfixeddirectiongeneral}
to show at each moment of time, the collection of vectors $V^{(0)}_1(t), V^{(1)}_2(t), \ldots,V^{(i_s)}_s(t)$ is independent uniform over $\dsphere.$ 
Furthermore, with high probability, for all $i,j,t,$ it holds that 
$$
|\langle V^{(i-1)}_j(t), V^{(i-1)}_i(t)\rangle - 
\langle V^{(i-1)}_j(t), V^{(i-1)}_i(t)\rangle|\le 
 C(L_t + U_t)\max\Big(
\frac{\log^{3/4} n\sqrt{np(L_t + U_t)}}{d^{1/4}}
,\frac{\log n}{\sqrt{d}}\Big)
$$
in analogy to \cref{lem:flippreservesinnerproducts}.
To prove this, one reasons in the exact same way as in \cref{lem:numberof} to show that each martingale differences sum has at most $nC\max( (L_t + U_t)\times p \times \sqrt{d}\times \sqrt{\log 1/p}, \log n)$ non-zero terms for some absolute constant $C.$ Furthermore, one can use the bounds from \cref{prop:elemntaryboundsgeneral} in place of \cref{prop:elemntarybounds}
to argue that each of the non-zero terms is bounded by $C(L_t + U_t)/d$ and the martingale-violation terms are bounded by 
$C(L_t + U_t)^2/d$ and $C(L_t + U_t)/d^{3/2}$ respectively.

One can similarly prove that 
\begin{align*}
& |\langle V^{(0)}_j(t+1), V^{(0)}_i(t+1)\rangle - 
\langle V^{(0)}_j(t), V^{(0)}_i(t)\rangle|\\
& \le 
(L_t + U_t) + 
 C(L_t + U_t)\max
 \Big(
\frac{\log^{3/4} n\sqrt{np(L_t + U_t)}}{d^{1/4}}
,\frac{\log n}{\sqrt{d}}\Big)
\end{align*}
when $t\ge 1.$

One concludes that 
with high probability, if $\langle V^{(0)}_j(t'), V^{(0)}_i(t')\rangle\not\in [\tau^p_d-L_{t'},\tau^p_d +U_{t'}]$ for some $t' \ge 1$ and $t'$ is the smallest index with this property, then $$H^{t'-1}_{ji}=\indicator[\langle V^{(0)}_j(t'), V^{(0)}_i(t')\rangle\ge\tau^p_d] = \indicator[\langle V^{(0)}_j(t), V^{(0)}_j(t)\rangle\ge\tau^p_d]\text{ for all $t>t'$ w.h.p.}$$  
Finally, \cref{eq:choiceofsymmetricintervals} implies that there will be no more defect edges after $T =  O(\frac{\log d}{\log \kappa} + \frac{\log n}{\log d})$ rounds. Indeed, note that if $I_t \coloneqq L_t + U_t$ is the length of the $t$-th interval,\footnote{We are intentionally loose with the log-factors here for simplicity of exposition.} then
$$
I_1 \le \log^3 n\times \frac{\sqrt{np}}{d} \quad\text{ and }
I_{t + 1}  
\le \log^2 n\times 
I_t \times 
\max\bigg(
\frac{\sqrt{np I_t}}{d^{1/4}}, 
\frac{1}{\sqrt{d}}
\bigg)
.
$$

Now, suppose that $d= (\log n)^6 np\times  \kappa$ for some $\kappa>1.$ By induction, trivially follow the following statements.
\begin{proposition} Suppose that $I_t \ge \frac{1}{d\sqrt{np}},$ so that $\frac{\sqrt{np I_t}}{d^{1/4}}\ge 
\frac{1}{\sqrt{d}}.$ Then, $I_t \le \frac{1}{\sqrt{d}}\Big(\frac{np\log^6n}{d}\Big)^{{\frac{3^{t-1}}{2^t}}}.$   
\end{proposition}
In particular, after $T_1 = O(\log_\kappa d) = O(\frac{\log d}{\log \kappa})$ rounds, it will be the case that $I_{T_1}\le \frac{1}{d\sqrt{np}} = \frac{1}{d^2}.$
\begin{proposition} Suppose that $I_t \le \frac{1}{d\sqrt{np}},$ so that $\frac{\sqrt{np I_t}}{d^{1/4}}\le 
\frac{1}{\sqrt{d}}.$ Then, $I_{t+s} \le I_t \times \Big(\frac{\log ^2 n}{d}\Big)^{s}.$   
\end{proposition}
The bounds on $F_t$ stated in \cref{obs:rggasrecursiveER} follow immediately from the above propositions combined with \cref{lemma:sphericalcdf}.
In particular, after $T_2 = \log_dn = O (\frac{\log n}{\log d})$ rounds, it will be the case that $I_{T_1 + T_2} \le \frac{1}{n^3\sqrt{d}}.$ By \cref{lemma:sphericalcdf}, $$\prob_{U,V\iidsim\unif(\dsphere)}\Big[X\in \big[\tau^p_d-\frac{1}{n^3\sqrt{d}}, \tau^p_d + \frac{1}{n^3\sqrt{d}}\big]\Big]\le C_D
    \frac{ p}{n^3}\sqrt{\log 1/p} = o(1/n^3).$$

Hence, after $T = O(\frac{\log d}{\log \kappa} + \frac{\log n}{\log d}),$ there would be no more defect edges which finishes the proof of \cref{obs:rggasrecursiveER}.